\documentclass[preprint,5p,times,twocolumn]{elsarticle}
\usepackage[T1]{fontenc}
\usepackage{xcolor}
\usepackage{hyperref}
\usepackage{amsthm}

\usepackage{amsmath,amssymb,amsthm}
\usepackage{graphicx}
\usepackage{float}
\usepackage{caption}
\usepackage{subcaption}
\usepackage{enumitem}
\usepackage{commath}
\usepackage{algorithm}
\usepackage{dsfont}
\usepackage{systeme}
\usepackage{tikz}
\usetikzlibrary{arrows.meta,positioning}

\allowdisplaybreaks

\usepackage{tikz}
\usetikzlibrary{shapes}

\def\createBlock#1#2#3#4{
	\tikzstyle{block}=[rectangle,draw, text centered, minimum height=3em, minimum width=3em, #4]
	\node[block] (#1) at (#2) {#3} ;
}

\def\createPoint#1#2{
	\coordinate (#1) at (#2) ;
}

\def\drawArrow#1#2#3#4#5 {
	\tikzstyle{arrow}=[->,>=latex]
	\draw[arrow] (#1) #3 (#2) node[midway,#5]{#4};
}

\tikzstyle{smallbox} = [
  rounded corners,
  align=center,
  inner sep=6pt,
  minimum height=2.5cm,
  minimum width=0.21\textwidth,
  fill=orange!20
]

\tikzstyle{sysbox} = [
  rounded corners,
  align=left,
  inner sep=6pt,
  minimum height=2.5cm,
  minimum width=0.45\textwidth,
   fill=blue!15
]

\tikzstyle{estbox} = [
  rounded corners,
  align=center,
  inner sep=6pt,
  minimum height=2.5cm,
  minimum width=0.1\textwidth,
  fill=green!20
]

\tikzset{
    myarrow/.style={->, line width=1mm, >=Stealth[scale=1.5]}
}

\DeclareCaptionSubType*[alph]{figure}
\newtheorem{thm}{Theorem}[section]
\newtheorem{assum}{Assumption}[section]
\newtheorem{prop}{Proposition}[section]
\newtheorem{defn}{Definition}[section]

\newtheorem{rem}{Remark}[section]

\newtheorem{lem}{Lemma}[section]
\newtheorem{claim}{Claim}[section]

\DeclareMathOperator*{\argmin}{arg\,min}

\newcommand{\R}{\mathbb{R}}
\newcommand{\N}{\mathbb{N}}

\newcommand{\xhat}{\hat{x}}
\newcommand{\zhat}{\hat{z}}
\newcommand{\vhat}{\hat{v}}

\newcommand{\dd}{\mathrm{d}}

\renewcommand{\epsilon}{\varepsilon}

\newcommand{\Tx}{T_0}
\newcommand{\Tv}{\mathbb{T}}
\newcommand{\Txv}{T}

\renewcommand{\leq}{\leqslant}
\renewcommand{\le}{\leqslant}
\renewcommand{\geq}{\geqslant}
\renewcommand{\ge}{\geqslant}

\DeclareMathOperator{\real}{Re}
\renewcommand{\Re}{\real}
\DeclareMathOperator{\imag}{Im}
\renewcommand{\Im}{\imag}

\newcommand{\Jinf}{J_{\infty}}

\newcommand{\dprod}{d_{\Pi}}
\newcommand{\dstar}{d_{\star}}

\newcommand{\fonction}[5]{
\begin{align*}
\displaystyle
\begin{array}{lrcl}
#1: & #2 & \longrightarrow & #3 \\
    & #4 & \longmapsto & #5
\end{array}
\end{align*}}

\def\createBlock#1#2#3#4{
    \tikzstyle{block}=[rectangle,draw,text centered,align=center,minimum height=3em,minimum width=3em,#4]
    \node[block] (#1) at (#2) {#3};
}

\def\drawArrow#1#2#3#4#5{
    \tikzstyle{arrow}=[->,>=latex]
    \draw[arrow] (#1) #3 (#2) node[midway,#5,align=center]{#4};
}

\newcommand{\modif}[1]{#1}
\begin{document}

\begin{frontmatter}

\title{A Backstepping-KKL Observer for a Cascade of a Nonlinear ODE with a Heat Equation%
\tnoteref{tnote1}}
\tnotetext[tnote1]{This project received funding from the Agence Nationale de la Recherche via grant PANOPLY ANR-23-CE48-0001-01.}

\author[inst1]{Adam Braun}
\ead{adam.braun@centralesupelec.fr}
\author[inst1]{Lucas Brivadis}
\ead{lucas.brivadis@centralesupelec.fr}
\author[inst1]{Jean Auriol}
\ead{jean.auriol@centralesupelec.fr}
\address[inst1]{Universit\'e Paris-Saclay, CNRS, CentraleSup\'elec, Laboratoire des Signaux et Syst\`emes, 91190 Gif-sur-Yvette, France}

\begin{abstract}
We propose an observer design for a cascaded system composed of an arbitrary nonlinear ordinary differential equation (ODE) with a 1D heat equation.
The nonlinear output of the ODE imposes a boundary condition on one side of the heat equation, while the measured output is on the other side.
The observer design combines an infinite-dimensional Kazantzis-Kravaris/Luenberger (KKL) observer for the ODE with a backstepping observer for the heat equation.
This construction is the first extension of the KKL methodology to infinite-dimensional systems. The dynamics are embedded into an appropriate target system via a Backstepping-KKL map. Under a differential observability condition on the ODE, this embedding is shown to be injective. \modif{Moreover, observer convergence is guaranteed under the additional assumption that the set of ODE trajectories of interest is compact.}
The effectiveness of the proposed approach is illustrated in numerical simulations.
\end{abstract}


\end{frontmatter}

\section{Introduction}
Parabolic partial differential equations (PDEs) naturally arise in the mathematical modeling of population dynamics~\cite{murray2002mathematical},~\cite{kolmogorov1937study}, heat conduction and diffusion processes~\cite{crank1979mathematics}, making the analysis of their controllability and observability fundamental topics in control theory.
The backstepping method was initially developed for parabolic PDEs to address control and observation problems. Since its inception, this approach has rapidly evolved and has been extended to the control of a broader class of PDEs. Comprehensive expositions of the method can be found in~\cite{krstic_book} and, more recently, in~\cite{VAZQUEZ2026112572}.
In recent years, significant progress has been made in applying the backstepping method to interconnected systems composed of linear Ordinary Differential Equations (ODEs) and Partial Differential Equations (PDEs). Such configurations are particularly relevant for modeling actuator and load dynamics~\cite{auriol2022comparing}. The backstepping approach was first applied to actuators governed by diffusion dynamics, as in~\cite{KRSTIC2009372}, where a stabilizing state-feedback control law and an observer were constructed for a cascade system in which a heat equation influences a linear ODE. This result was later generalized in~\cite{TANG20112142} to cover systems where the ODE also acts on the heat equation.
The method was subsequently extended to other classes of ODE-PDE interconnections, including hyperbolic PDEs, as in~\cite{FERRANTE2020109027}, where a Lyapunov matrix inequality approach is employed, and in~\cite{deutscher2017outputfeedbackcontrolgeneral,auriol2022observer}. More complex configurations have also been investigated, such as systems where a PDE is sandwiched between two ODEs, with results for the heat equation presented in~\cite{wang2019outputfeedbackboundarycontrolheat} and for hyperbolic PDEs in~\cite{DEUTSCHER2018472,redaud2024output}.
Concerning cascades of ODEs with linear PDEs,
the contributions in~\cite{ode_pde_Giri,Cai2016ObserverDesignCascade}, as well as the adaptive extension in~\cite{Benabdelhadi2023AdaptiveObserver}, propose observer designs that combine high-gain techniques for the nonlinear ODE with backstepping methods for the PDE. These works have shown that such a strategy can successfully reconstruct the state of the overall ODE-PDE cascade and establish convergence of the estimation error.
\modif{However, these approaches require structural assumptions allowing the ODE dynamics to be represented in an appropriate triangular form~\cite{bernard2019observer}. Moreover, depending on the considered measurement configuration, additional restrictions involving the PDE dynamics or domain length may arise. We investigate here an alternative construction based on the KKL methodology, which relaxes the triangular-form requirement on the ODE, at the cost of introducing the problem of inverting a nonlinear Backstepping-KKL mapping.}
\modif{This does not, however, rule out transient amplification resulting from the inversion of the KKL mapping.}

The KKL observer was originally introduced in~\cite{kazantzis1998nonlinear, andrieu2006existence}, as an extension to nonlinear systems of the seminal Luenberger methodology developed for linear systems in~\cite{Luenberger64}. A recent presentation of this theory can be found in~\cite{further_remarks}.
The KKL framework applies to any nonlinear ODEs, without structural restrictions. The price of this generality is that computing the observer requires the inversion of a nonlinear map, commonly referred to as the KKL mapping, which is generally difficult to characterize explicitly or compute. Nevertheless, several numerical methods have been proposed to approximate this mapping and implement the corresponding observer in practice. In particular, recent works have shown that efficient approximations can be obtained using neural-network-based approaches, see for instance~\cite{deepKKL,BuissonFenet2023RecognitionModels,shaaban2026hyperkkl}.
Since the observer proposed in this article can be interpreted as an infinite-dimensional extension of the KKL observer, a possible direction for future research is to adapt these numerical methods to the high-dimensional setting.




In the present paper, we rather focus on designing the observer \modif{and proving its convergence towards the estimated state}.
We consider a cascade system formed by a nonlinear ODE (with state $x(t)$ in $\R^n$) and a one-dimensional heat equation (with state $v$ in $L^2(0,1)$). The ODE influences the PDE through a Dirichlet boundary condition ($v(t,1) = h(x(t))$), while the output $y$ of the overall cascade system is measured at the opposite boundary ($y=v(t,0)$) as depicted in Figure~\ref{Dessin cascade}.

\begin{figure}[ht!]
    \centering
    \begin{tikzpicture}[scale=0.75, transform shape]
		
    
    \createBlock{ODE}{-18em,0}{Nonlinear ODE $\dot x = f(x)$}{text width = 9em};
    \createBlock{PDE}{-1em,0}{
    1D heat equation
    }{text width = 9em};
    \createPoint{out}{10em,0} ;
    \drawArrow{ODE}{PDE}{--}{$v(t,1)=h(x)$}{above} ;
    \drawArrow{PDE}{out}{--}{$\quad y=v(t,0)$}{above} ;
    \end{tikzpicture}
     \caption{Representation of the cascaded system}
    \label{Dessin cascade}
\end{figure}

The observer is designed  by employing a KKL observer for the nonlinear ODE and a backstepping observer for the linear PDE. The motivation for this work originates from the observation of structural similarities between backstepping observers and KKL observers in the context of linear dynamics, please refer to Section~\ref{KKL dim finie} for more details.
We will adopt the methodology of~\cite[Section 4]{further_remarks}, which addresses observer design for cascaded systems consisting of a nonlinear ODE followed by a linear ODE in normal form.

The structure of this article is the following. In Section~\ref{problem statement}, we present the cascade system and discuss its well-posedness. In Section~\ref{KKL dim finie}, we review recent developments on KKL observers in finite dimension. In Section~\ref{Observer naive},
we formally design our Backstepping-KKL observer, obtained by inverting a map $T$ which is the sum of a backstepping transform $\Tv$ and a KKL transform $T_0$. In Section~\ref{construction of T0}, we show the existence of $T_0$. In Section~\ref{injectivity of KKL}, we prove the injectivity of $T$ under a differential observability assumption and we state the main result of the article, which is the convergence of the observer towards the state under that observability assumption. In Section~\ref{numerics}, we present several examples corresponding to specific choices of ordinary differential equations, along with numerical simulations that illustrate the performance of the proposed observer design.
\paragraph{Notation}
Throughout this article, we use the following notations.
\begin{itemize}
    \item $L^2(0,1)$ is the  space of real-valued square integrable functions defined almost everywhere on $(0,1)$. The associated norm is denoted $\|\cdot\|_{L^2}$.
    \item $C^m(I, \mathcal{B})$ denotes the set of functions from $I\subset\R$ to the Banach space $\mathcal{B}$ that are $m$ times continuously Fréchet differentiable.
    \modif{\item $H^m(0,1)$ denotes the Sobolev space of real-valued functions on $(0,1)$ whose weak derivatives up to order $m$ belong to $L^2(0,1)$. Its norm is denoted $\|\cdot\|_{H^m}.$
    \item $H^\infty(0,1):=\bigcap_{m\in\mathbb N}H^m(0,1)$ is endowed with the topology of convergence in every Sobolev norm $\|\cdot\|_{H^m}$.}
    \item $|\cdot|$ is the euclidean norm on $\R^n$.
    \item For any vector field $f:\R^n\to\R^n$ and any map $h:\R^n\to\R^m$, we denote by $L_f^kh$ the $k$-th iterate Lie derivative of $h$ with respect to $f$ when it makes sense.
\end{itemize}

\section{Problem statement and well-posedness} \label{problem statement}
Consider the following ODE-PDE cascade system
\begin{equation}\label{eq:syst}
\left\{
\begin{aligned}
   & \dot x(t) = f(x(t))
    \\
    &\partial_tv(t, \lambda) = \partial_{\lambda}^2 v(t, \lambda) + \alpha v(t, \lambda)
    \\
    &\partial_{\lambda}v(t, 0) = 0
    \\
    &v(t, 1) = h(x(t))
    \\
   & y(t) = v(t, 0),
\end{aligned}
\right.
\end{equation}
where $v$ defined on $[0,\infty)\times[0,1]$ is the state of the 1D heat equation and $x$ defined on  $[0,\infty)$ is the ODE state, $y:\R_+\to \R$ is the scalar measured output, $f:\R^n\to\R^n$ and $h:\R^n\to\R$ are nonlinear maps, and $\alpha \in \mathbb{R}$.
Let $\mathcal{X}_0 \subset \R^n$ be an open set, representing initial conditions of interest.
We assume that the Cauchy problem $(\dot x = f(x),\ x(0)=x_0)$ admits a unique forward complete solution for any $x_0\in\mathcal{X}_0$, denoted by $t\mapsto X(t;x_0)$.
Let $\mathcal{X}\subset\R^n$ be the set containing the trajectories of interest of the ODE, i.e $\mathcal{X} = \{ X(t;x_0): x_0\in \mathcal{X}_0, ~t\geq0\}$.
We assume that $h$ is continuous, so that the 1D heat equation appearing in \eqref{eq:syst} has homogeneous Neumann condition at one boundary and nonhomogeneous Dirichlet condition at the other (see e.g. \cite[Chapter~7]{evans10} for well-posedness).
Therefore, for all $x_0 \in \mathcal{X}_0$ and all $v_0 \in L^2(0,1)$, there exists a unique solution $(x,v)\in C^0(\R_+, \R^n\times L^2(0,1))$ to the cascade system~\eqref{eq:syst} satisfying the initial conditions
$
x(0) = x_0, v(0,\cdot) = v_0.
$


In the following, for any trajectory of interest $(x,v)$ of the cascade system~\eqref{eq:syst}, we will design an observer  based on the sole online measurement of $y$. More precisely,
we will design
$(\xhat(t), \vhat(t))$ satisfying
\begin{equation*}
    \left | \xhat(t) - x(t)\right | + \|\vhat(t) - v(t)\|_{L^2} \underset{t\to+\infty}{\longrightarrow} 0.
\end{equation*}
Our observer design relies on the standard differential observability assumption for the ODE (see e.g \cite[Definition 4.2, Part 1, Section 2.4]{Gauthier_Kupka_2001}).
\begin{assum}[Differential observability]\label{diff_obs}
    There exists a non-negative integer $m$ such that the function
    \fonction{H}{\mathcal{X}}{\R^m}{x}{(h(x), L_f h(x), \hdots, L_f^{m-1}h(x))}
exists and
is injective.
\end{assum}
\begin{rem}
Under the additional assumption that $H$ is an immersion, it transforms the original ODE into a triangular form in which a high-gain observer can be designed for system \eqref{eq:syst} in the manner of \cite{ode_pde_Giri,Cai2016ObserverDesignCascade}.
\end{rem}
In addition, we impose the following auxiliary assumption, which can be readily verified in practice (see Remark~\ref{traj negative}).
\begin{assum}\label{ass:bound}
Either
\begin{enumerate}
    \item $\mathcal X$ is bounded; \emph{or}
    \item there exists $\gamma_0>0$ such that, for all $x\in \mathcal X_0$, $X(t;x)$ is well defined for all $t\le 0$, and there exists $C=C(x)>0$ for which
    \begin{equation} \label{ineq:gamma0}
    e^{\gamma_0 t}\lvert h(X(t;x))\rvert \le C,
    \quad \forall t\le 0.
    \end{equation}
\end{enumerate} 
\end{assum}
When the first item of Assumption~\ref{ass:bound} is satisfied, we set $\gamma_0 = 0$ in the following.

\begin{rem}
Assuming that $\mathcal X$ is bounded holds 
in general for the purpose of observer design, as we are concerned in estimating trajectories of interest remaining in a given bounded set \cite{bernard2022observer}.
    \label{traj negative}
    The second alternative is satisfied if either
    it holds that $|f(x)|\leq f_0+f_1|x|$
    and $|h(x)|\leq h_0+h_1|x|^q$ for all $x\in\R^n$ for some positive constants $f_0, f_1, h_0, h_1, q$,
    or if 
    $X(t;x)$ is well defined for all $t<0$ and $h$ is bounded.
\end{rem}

\section{Link between the KKL and backstepping observers}
\label{KKL dim finie}

In this section, we recall the foundations of the KKL methodology in its original finite-dimensional setting and highlight conceptual parallels with the backstepping method for parabolic PDEs.
 Let us consider a system of the form
\begin{equation}\label{meta ODE}\dot{x} =f(x)~, ~y=h(x),\end{equation}
where $f:\R^n\to \R^n$ and $h:\R^n\to \R$ are smooth functions, $x$ is the state of the system and $y$ the output. If there exists a map $T:\R^n\to \R^m$ (with $m$ an integer typically larger than $n$)  such that $z=Tx$ (with $x$ solution of~\eqref{meta ODE}) satisfies 
\begin{equation}
    \label{meta dynamique stable}
\dot{z} = Az + By,\end{equation} with $A\in \R^{m\times m}$ an Hurwitz matrix and $B\in \R^m$,
and if $T$ admits a uniformly continuous left inverse $T^{-1}:\R^m\to\R^n$,
then a KKL observer is given by $\hat{x} = T^{-1}(\zhat)$ with $\zhat$ any solution of~\eqref{meta dynamique stable}. Indeed, $\tilde{z} = \zhat-z$ satisfies the contracting dynamic $\tilde{z} = A\tilde{z}$, hence $\zhat$ converges to $z$ as $t\to \infty$ and using the uniform continuity of $T^{-1}$, $\xhat$ converges to $x$ as $t\to \infty$. 
The KKL observer design methodology is summarized in Figure~\ref{schema:KKL}.
\begin{figure}[H]
\centering
\begin{tikzpicture}[scale=0.75, transform shape]

    \createBlock{in}{0,0}{$x(t),\,y_0(t)$}{text width = 10em}

    \createBlock{kkl}{8,0}{
        {\bfseries KKL Target system}\\[0.15cm]
        $\dot z = Az + By_0$
    }{text width = 10em}

    \createBlock{xhat}{0,-2}{$\hat x(t)$}{text width = 10em}

    \createBlock{zhat}{8.0,-2}{$\dot{\hat z} = A\hat z + By_0$}{text width = 10em}

    \drawArrow{in}{kkl}{--}{KKL mapping $T_0$}{above}
    \drawArrow{zhat}{xhat}{--}{Left Inverse $T_0^{-1}$}{above}

\end{tikzpicture}
\caption{Flow chart for the KKL observer design.}
\label{schema:KKL}
\end{figure}

The injectivity of $T$ is generally proved under an observability assumption like the backward distinguishability (as in~\cite{further_remarks}) or like in our case with Assumption~\ref{diff_obs}, the differential observability (as in~\cite{bernard2022observer}).

It is worth emphasizing that the design of an observer for a linear PDE, such as the one-dimensional heat equation, via a backstepping transformation as presented in~\cite[Chapter~5]{krstic_book}, is conceptually equivalent to the construction of a Kazantzis--Kravaris--Luenberger (KKL) observer.
 In this setting, the backstepping transformation plays a role analogous to that of the operator~$T$ introduced earlier. In~\cite[Chapter 5]{krstic_book}, the PDE observer is built as a copy of the plant dynamics augmented by a correction term proportional to the discrepancy between the outputs of the observer and the system. The invertible backstepping transformation is then used to map the estimation error dynamics into an exponentially stable heat equation.
Because the system is linear, one could equivalently apply the backstepping transformation to both the system and the observer from the outset, and then analyze the estimation error in the transformed coordinates, again yielding an exponentially stable heat equation. This viewpoint closely parallels the use of the mapping~$T$ in the construction of a KKL observer.
The previous explanation is summarized in Figure~\ref{schema:backstepping}.

\begin{figure}[H]
\centering
\begin{tikzpicture}[scale=0.75, transform shape]

    \createBlock{in}{0,0}{$v(t),\,y(t)$}{text width = 10em}

    \createBlock{target}{7,0}{
        {\bfseries Backstepping Target system}\\[0.15cm]
        $\partial_t z(t,\lambda)=\partial_\lambda^2 z(t,\lambda)-\gamma z(t,\lambda)+p_1(\lambda)y(t)$\\[0.15cm]
        $\partial_\lambda z(t,0)=p_{10}y(t)$\\
        $z(t,1)=u(t)$
    }{}

    \createBlock{vhat}{0,-3.0}{$\hat v(t)$}{text width = 10em}

    \createBlock{zhat}{7,-3.0}{
        $\partial_t \hat z(t,\lambda)=\partial_\lambda^2 \hat z(t,\lambda)-\gamma \hat z(t,\lambda)+p_1(\lambda)y(t)$\\[0.15cm]
        $\partial_\lambda \hat z(t,0)=p_{10}y(t)$\\
        $\hat z(t,1)=u(t)$
    }{text width = 10em}

    \drawArrow{in}{target}{--}{$\Tv$}{above}
    \drawArrow{zhat}{vhat}{--}{$\Tv^{-1}$}{above}

\end{tikzpicture}
\caption{Flow chart for the backstepping observer design.}
\label{schema:backstepping}
\end{figure}

\section{Observer design using a formal  approach }
\label{Observer naive}
Following the outline of the KKL construction in the finite-dimensional case, we aim to define a mapping $T$ that sends the state $(x,v)$ of the cascade~\eqref{eq:syst} to a KKL target system of the form~\eqref{meta dynamique stable}. The main difference is that we propose the KKL target system to be a PDE, with the mapping $T$ defined as $T(x,v) = T_0(x) + \Tv(v)$ where $\Tv$ is a classical backstepping transformation used to design an observer for an unstable heat equation, and $T_0$ is an infinite-dimensional KKL mapping associated with the ODE.
Adapting~\cite[Chapter 5]{krstic_book}, we define the following invertible backstepping transformation $\Tv$, mapping any solution $v$ of the PDE in~\eqref{eq:syst}
to a solution $z_{pde}=\Tv(v)$ to the PDE-target system~\eqref{PDE-Target_system}:
    \fonction{\Tv}{L^2(0, 1)}{L^2(0, 1)}{v}{\Big(\lambda\mapsto v(\lambda)-\int_0^\lambda p(\lambda, \tilde \lambda)v(\tilde \lambda)\dd \tilde \lambda \Big)}

where $p\in L^2((0, 1)^2)$ satisfies the following kernel equations:
\begin{equation}
\begin{aligned}\label{kernel equations p}
   & p_{\lambda \lambda}(\lambda, \tilde{\lambda})-p_{\tilde{\lambda}\tilde{\lambda}}(\lambda, \tilde{\lambda}) = (\alpha+\gamma) p(\lambda, \tilde{\lambda})\\
    &\frac{dp}{d\lambda}(\lambda, \lambda) = -(\alpha + \gamma)/2\\
    &p(1, \tilde{\lambda})=0,
\end{aligned}
\end{equation}
and $\gamma > \gamma_0$ with $\alpha +\gamma \neq 0$ is a tuning  parameter. \modif{In particular, we always have \(\gamma>0\), which will guarantee the contraction of the KKL target system.} We stress that $\gamma$ does not play the role of a high-gain parameter. Rather, it only needs to satisfy $\gamma>\gamma_0$, ($\gamma_0=0$ when the first item of Assumption~\eqref{ass:bound} is satisfied).
Using the method of successive approximations~\cite{krstic_book}, we can explicitly compute the backstepping kernel and prove that the following $p$ is the unique solution to the system of equations~\eqref{kernel equations p}
\[
p(\lambda, \tilde{\lambda}) = (\alpha + \gamma)(1 - \lambda) \cdot 
\frac{J_1\Bigl(\sqrt{(\alpha + \gamma)\bigl(\tilde{\lambda}^2 - \lambda^2 + 2(\lambda - \tilde{\lambda})\bigr)}\Bigr)}
{\sqrt{(\alpha + \gamma)\bigl(\tilde{\lambda}^2 - \lambda^2 + 2(\lambda - \tilde{\lambda})\bigr)}},
\]
where $J_1$ is a Bessel function of the first kind.
Clearly, $p$ is analytic since $J_1(z)/z\to 1/2$ as $z\to0$.
The PDE-target system is defined for all $t\geq 0$ and all $\lambda \in [0,1]$ by 
\begin{equation}\label{PDE-Target_system}
\left\{
\begin{aligned}
    &\partial_t z_{\textrm{pde}}(t, \lambda) = \partial_{\lambda}^2 z_{\textrm{pde}}(t, \lambda)-\gamma z_{\textrm{pde}}(t,\lambda) + p_1(\lambda)y(t) 
    \\
    &\partial_{\lambda}z_{\textrm{pde}}(t, 0) = p_{10}y(t)
    \\
    &z_{\textrm{pde}}(t, 1) = h(x(t)),
\end{aligned}
\right. 
\end{equation}
with $p_1(\lambda):=-\partial_{\tilde{\lambda}}p(\lambda,0)$ and $p_{10} = -p(0,0) = -\frac{1}{2}(\alpha + \gamma)\neq 0$. 
\modif{We have the following regularity Lemma of which the proof is straightforward since $\Tv^{-1}$ is also a Volterra operator~\cite{yoshida1960lectures}.
\begin{lem}
\label{lem:inverse_backstepping_regularity}
\begin{enumerate}
    \item \(\Tv\) and \(\Tv^{-1}\) map analytic functions on \([0,1]\)
    to analytic functions on \([0,1]\);
    \item for all integer \(m\ge0\), \(\Tv\) and \(\Tv^{-1}\) are bounded
    operators from \(H^m(0,1)\) into itself.
\end{enumerate}
\end{lem}}
We now define the KKL target system by removing the unknown term $h(x(t))$ from the Dirichlet boundary condition at $\lambda = 1$. We have
\begin{equation}\label{KKL-target system}
\left\{
\begin{aligned}
    &\partial_tz(t, \lambda) = \partial_{\lambda}^2z(t, \lambda)-\gamma z(t,\lambda) + p_1(\lambda)y(t) 
    \\
    &\partial_{\lambda}z(t, 0) = p_{10}y(t)
    \\
    &z(t, 1) = 0.
\end{aligned}
\right.
\end{equation}
 Observe that system~\eqref{KKL-target system} shares the structure of~\eqref{meta dynamique stable} where $A$ is now an operator defined on $D(A):=\{z\in H^2(0, 1)\mid \partial_ {\lambda}z(0) = 0 \text{ and } z(1) = 0\}$ by $Az = \partial_{\lambda}^2z -\gamma z$. The operator $A$ induces a contracting dynamic ($\gamma$ is a degree of freedom introduced to accelerate the contraction, see for instance~\cite{krstic_book}), and $B$ is the operator acting both on the domain (with  $p_1(\cdot)$) and on the boundary at $\lambda=0$ (with  $p_{10}$). It is defined similarly to~\cite[Chapter 10]{tucsnak2009observation}.
Now, assume there exists a mapping \(T_0\) such that  
\[
z = T(x,v) \;=\; T_0(x) + \Tv(v),
\]
is a solution of~\eqref{KKL-target system} whenever \((x,v)\) is a solution of the cascade system~\eqref{eq:syst}.  
Then, as in the finite-dimensional case, if \(T\) is invertible with a uniformly continuous inverse, an observer can be defined as  
\[
(\hat{x}, \hat{v}) \;=\; T^{-1}(\hat{z}),
\]
where \(\hat{z}\) denotes any solution of~\eqref{KKL-target system}.
\modif{This formal construction motivates the observer design. As shown in
Section~\ref{sec:main_result}, a continuous left inverse does not exist for the
$L^2(0,1)$ topology. The actual reconstruction therefore requires a
stronger topology, together with a minimization procedure that remains
well defined when the observer state does not belong to the exact image
of the Backstepping-KKL transformation.}
The operator $T_0$ we are looking for needs to follow a heat equation like \eqref{KKL-target system} along the trajectories of the ODE of the cascade~\eqref{eq:syst} and needs to cancel the boundary term $h(x(t))$.
Hence we want to ensure the existence of a map $T_0$ such that for all $x \in \mathcal{X}_0$, all $t\geq 0$ and all $\lambda \in [0,1]$,
\begin{equation}\label{T_0 naif}
\left\{
\begin{aligned}
    &\partial_t T_0(X(t;x), \lambda) = \partial_{\lambda}^2T_0(X(t;x), \lambda)-\gamma T_0(X(t;x),\lambda)  
    \\
    &\partial_{\lambda}T_0(X(t;x), 0) = 0
    \\
    &T_0(X(t;x), 1) = -h(X(t;x)).
\end{aligned}
\right.
\end{equation}
The observer design methodology is summarized in Figure~\ref{fig:flow_chart}

\begin{figure}[H]
\centering
\begin{tikzpicture}[scale=0.75, transform shape]

\createBlock{signals}{0,0}{
$x(t),\; v(t)$ \\[4pt]
$y(t)=v(t,0)$
}{
minimum width=3.0cm,
minimum height=1.5cm,
align=center
}

\createBlock{pdesys}{7,0}{
{\bfseries PDE-Target system}\\[6pt]
$\displaystyle
\begin{aligned}
\dot{x}(t) &= f(x(t)) \\
\partial_t z_{\mathrm{pde}}(t,\lambda)
&= \partial_{\lambda}^2 z_{\mathrm{pde}}(t,\lambda)
   - \gamma z_{\mathrm{pde}}(t,\lambda)
   + p_1(\lambda)\,y(t) \\
\partial_{\lambda} z_{\mathrm{pde}}(t,0) &= p_{10}\,y(t) \\
z_{\mathrm{pde}}(t,1) &= h(x(t))
\end{aligned}
$
}{
align=center
}

\createBlock{kklsys}{7,-6}{
{\bfseries KKL-target system}\\[6pt]
$\displaystyle
\begin{aligned}
\partial_t z(t,\lambda)
&= \partial_{\lambda}^2 z(t,\lambda)
   - \gamma z(t,\lambda)
   + p_1(\lambda)\,y(t) \\
\partial_{\lambda} z(t,0) &= p_{10}\,y(t) \\
z(t,1) &= 0
\end{aligned}
$
}{
align=center
}

\createBlock{estimation}{0,-9}{
$\hat x(t),\; \hat v(t)$
}{
minimum width=3.0cm,
minimum height=1.5cm,
align=center
}

\createBlock{kklsys_hat}{7,-9}{
{\bfseries Estimated KKL system}\\[6pt]
$\displaystyle
\begin{aligned}
\partial_t \hat z(t,\lambda)
&= \partial_{\lambda}^2 \hat z(t,\lambda)
   - \gamma \hat z(t,\lambda)
   + p_1(\lambda)\,y(t) \\
\partial_{\lambda} \hat z(t,0) &= p_{10}\,y(t) \\
\hat z(t,1) &= 0
\end{aligned}
$
}{
align=center
}

\drawArrow{signals.east}{pdesys.west}{--}{$\Tv$}{above}

\drawArrow{pdesys.south}{kklsys.north}{--}{$+\;$KKL $\Tx$}{right}

\draw[->,>=latex] (signals.south east) -- (kklsys.north west)
    node[midway, sloped, above, align=center] {Backstepping-KKL}
    node[midway, sloped, below, align=center] {$T(x,v)=\Tv(v)+\Tx(x)$};

\drawArrow{kklsys_hat.west}{estimation.east}{--}{\textbf{Left Inverse} $T^{-1}$}{above}

\end{tikzpicture}
\caption{Flow chart for the observer design.}
\label{fig:flow_chart}
\end{figure}
\begin{rem}
One can notice that if it exists, the operator $x \mapsto T_0(x, \cdot)$ maps any solution $x$ of the ODE in the cascade~\eqref{eq:syst} to a solution $z_\mathrm{ode} = T_0(x,\cdot)$ of the target system defined for all $t\geq 0$, and all $\lambda \in [0,1]$ by
    \begin{equation}\label{eq:z_ode}
\left\{
\begin{aligned}
    &\partial_t z_\mathrm{ode}(t, \lambda) = \partial_{\lambda}^2 z_\mathrm{ode}(t, \lambda)-\gamma  z_\mathrm{ode}(t, \lambda)
    \\
    &\partial_{\lambda} z_\mathrm{ode}(t, 0) = 0
    \\
    & z_\mathrm{ode}(t, 1) = -h(x(t)),
\end{aligned}
\right. 
\end{equation}
so that the sum of a solution to \eqref{PDE-Target_system} and a solution to \eqref{eq:z_ode} is indeed a solution to \eqref{KKL-target system}.
\end{rem}

\begin{rem}\label{notation lambda}
    The choice of the notation $\lambda$ for the spatial variable is motivated by the finite-dimensional KKL framework, where the KKL mapping is obtained by solving a Sylvester equation in a basis in which the matrix $A$ of the target system~\eqref{meta dynamique stable} is diagonal, with eigenvalues $\lambda_i$. In the infinite-dimensional generalization of the KKL approach, the steady-state problem plays the role of the Sylvester equation, and the spatial variable $\lambda$ serves as the analogue of the eigenvalues $\lambda_i$.
\end{rem}
    
\section{Existence of \texorpdfstring{$T_0$}{T0}}
\label{construction of T0}

In this section, we prove the existence of the operator $T_0$ such that~\eqref{T_0 naif} is satisfied.
To this end, let us consider the following auxiliary heat equation problem without initial condition. For $\gamma >\gamma_0$, for any continuous input $u$, for all $t\leq 0$, for all $\lambda \in [0,1]$, we consider the PDE system
\begin{equation}\label{stable heat}\begin{cases}
    &\partial_t w_u(t,\lambda) = \partial_{\lambda}^2 w_u(t,\lambda) - \gamma w_u(t,\lambda)\\
    &\partial_{\lambda}w_u(t,0)=0 \\
    &w_u(t,1) = u(t).
\end{cases}\end{equation}

This so called \emph{steady state} problem was initially investigated by Tikhonov in~\cite[Chapter III, section 3.4]{TikhonovSamarskii1990}, where he established the existence and uniqueness of a bounded solution in the case of an infinite space domain, $\lambda \in [0, \infty)$. 
In the following lemma, using Tikhonov's result, we demonstrate the existence and the uniqueness of bounded solution of~\eqref{stable heat} in the classical sense.
\begin{lem}\label{pb stable heat bien posé}
  for all \( u \in C^{0}((-\infty, 0])  \) such that $t \to e^{\gamma t}u(t)$ is bounded on $(-\infty, 0]$, the problem~\eqref{stable heat} admits a unique solution \[ w_u \in C^1((-\infty, 0], C^2([0,1])) \] such that there exists $C>0$ such that for all $t\leq0$ 
$$
\|w_u(t,\cdot)\|_{L^2(0,1)} < C e^{-\gamma t}.
$$ Moreover, for all $t\leq0$, $w_u(t,\cdot)$ is an element of $\mathcal{A}$, where $\mathcal{A}$ is the space of analytical functions $v$ on $[0,1)$ such that for all $k\in \N$, 
    $\frac{d^{2k+1}v}{d \lambda^{2k+1}}(0)=0$.
\end{lem}
\noindent
The proof is postponed to~\ref{tech lemma 1}.

We now define an infinite-dimensional KKL transformation $T_0$ so as to satisfy~\eqref{T_0 naif}, relying on Lemma~\ref{pb stable heat bien posé}.
To this end, it is required that, for all $x\in\mathcal X_0$, the trajectory $X(t;x)$ be defined for all $t\leq 0$, and that the function
\[
u(t):=-h(X(t;x)),
\]
is such that $e^{\gamma t}u(t)$ is bounded on $(-\infty,0)$.
These conditions are immediately fulfilled under the second condition of Assumption~\ref{ass:bound}. When the first condition is satisfied, inspired by~\cite{further_remarks}, we replace $f$ with the modified vector field $\chi(x)f(x)$, where $\chi(x)f(x)=f(x)$ for all $x\in\mathcal X$, and $\chi=0$ outside an open set containing the closure of $\mathcal X$. This modification leaves the trajectories of interest unchanged in forward time, while ensuring that they remain in a compact set in backward time. Consequently, $X(t;x)$ is well defined and $u(t) = -h(X(t;x))$ is bounded for all $t\leq0$.
In what follows, this modified vector field is still denoted by $f$.
\begin{defn}\label{def T0}
Suppose Assumption~\ref{ass:bound} holds. Let $x\in \mathcal X$.
Let $w_{u_x}$ be the unique solution of~\eqref{stable heat} with, for all $t\leq 0$, $u(t) = u_x(t) := -h(X(t;x))$ (this choice of $u$ is valid because of Assumption~\ref{ass:bound}).
    We define for all $x\in \mathcal{X}$, for all $\lambda\in [0,1]$, $$T_0(x,\lambda): = w_{u_x}(0,\lambda).$$
\end{defn}
A \emph{direct consequence} of Definition~\ref{def T0} is that for all $x\in \mathcal{X}$, $T_0(x,\cdot)\in \mathcal{A}$.
The next proposition shows that $T_0$ is indeed a solution to a heat equation along the flow of the ODE in the cascade~\eqref{eq:syst}, i.e $T_0$ is a solution of system~\eqref{T_0 naif}.
\begin{prop}\label{T_0alongflow} Using the same notations as in Definition~\ref{def T0}, the KKL mapping $\Tx$ satisfies, for all $x_0\in\mathcal{X}_0$, all $\tau\geq0$,
   and all $0\leq t\leq \tau$, 
        \[
    T_0(X(t;x_0), \lambda) = 
        w_{u_{x}}(t-\tau,\lambda),
\]
with $x := X(\tau;x_0)$.
\end{prop}
\begin{proof}

\medskip

\noindent\textbf{Step 1: The case $\tau= 0$.} Let $x \in \mathcal{X}$. By the flow property of the ODE in~\eqref{eq:syst} (using Assumption~\ref{ass:bound}), for all $t \leq 0$, and all $s \leq 0$, we have
\[
    u_{X(t;x)}(s) 
    = -h(X(s;X(t;x))) 
    = -h(X(s+t;x)) 
    = u_x(s+t).
\]
Consequently,
\begin{equation}\label{eq:flow_negative}
    T_0(X(t;x)) 
    = w_{u_{X(t;x)}}(0) 
    = w_{u_x(\cdot+t)}(0) 
    = w_{u_x}(t),
\end{equation}
where the last equality follows from the uniqueness of the bounded solution of~\eqref{stable heat}, as established in Lemma~\ref{pb stable heat bien posé}.

\medskip

\noindent\textbf{Step 2: The case $\tau\geq0$.}
Let $x_0 \in \mathcal{X}_0$.
The flow property of the cascade system~\eqref{eq:syst} yields
\begin{equation}\label{eq:flow_positive}
    X(t;x_0) = X(t-\tau;X(\tau;x_0)), \quad \forall t \leq \tau.
\end{equation}
Let $x := X(\tau;x_0) \in \mathcal X$ and set $s := t-\tau \leq 0$.  
Combining~\eqref{eq:flow_negative} with~\eqref{eq:flow_positive}, we obtain
\[
    T_0(X(t;x_0)) 
    = T_0(X(s;x)) 
    = w_{u_{x}}(s) 
    = w_{u_{x}}(t-\tau),
\]
which completes the proof.
\end{proof}

\modif{Proposition~\ref{T_0alongflow} expresses the shift property of $T_0$ construction along the flow of the ODE.}
Since the objective of this article is to construct an observer converging to the state of the cascade system~\eqref{eq:syst}, the next section is devoted to analyzing the injectivity of the now well-defined Backstepping-KKL transformation.
\section{Injectivity of the Backstepping-KKL transformation and convergence of the observer}
\label{injectivity of KKL}
As explained in Section~\ref{Observer naive}, the observer for the cascade~\eqref{eq:syst} will be constructed applying the inverse of  a Backstepping-KKL mapping $T$ to a solution of the KKL-target system~\eqref{KKL-target system}. In this section we prove the injectivity of $T$
and to do so we first prove the injectivity of $x\to T_0(x,\cdot)$. The main result of this article is stated at the end of the Section.
\subsection{Injectivity of \texorpdfstring{$T_0$}{T0}}
\begin{thm}\label{T0 inj}
    Under Assumptions~\ref{diff_obs} and~\ref{ass:bound}, the operator
    \fonction{T_0}{\mathcal{X}}{L^2(0,1)}{x}{\lambda\to T_0(x,\lambda)} is injective.
\end{thm}
\begin{proof}
Let $x_a,x_b\in\mathcal X$ be such that $T_0(x_a)=T_0(x_b)$ in $L^2(0,1)$.  
By the parabolic regularity of the heat equation~\cite[Chapter~7]{evans10}, the functions
\[
\lambda\mapsto T_0(x_a,\lambda),\quad \lambda\mapsto T_0(x_b,\lambda)
\]
are continuous on $[0,1]$. Hence the equality actually holds for all
$\lambda\in[0,1]$, that is,
\begin{equation}\label{eq:T0-equal}
w_{u_{x_a}}(0,\lambda)=w_{u_{x_b}}(0,\lambda),
\quad \forall \lambda\in[0,1],
\end{equation}
where
\[
u_{x_a}(s):=-h(X(s;x_a)),
\quad
u_{x_b}(s):=-h(X(s;x_b)).
\]
We claim that for all $k\in\{0,\dots,m-1\}$,
\begin{equation}\label{eq:claim-k}
\partial_t^k w_{u_{x_a}}(0,\lambda)
=
\partial_t^k w_{u_{x_b}}(0,\lambda),
\quad \forall \lambda\in[0,1].
\end{equation}

We prove \eqref{eq:claim-k} by induction on $k$.
For $k=0$, \eqref{eq:claim-k} is exactly \eqref{eq:T0-equal}.
Assume now that \eqref{eq:claim-k} holds for some
$k\in\{0,\dots,m-2\}$. Since the solutions are smooth with respect to
$(t,\lambda)$, time and space derivatives commute. Differentiating
\eqref{eq:claim-k} twice with respect to $\lambda$, we obtain
\[
\partial_t^k\partial_\lambda^2 w_{u_{x_a}}(0,\lambda)
=
\partial_t^k\partial_\lambda^2 w_{u_{x_b}}(0,\lambda),
\quad \forall \lambda\in[0,1].
\]
Using
\[
\partial_t w=\partial_\lambda^2 w-\gamma w,
\]
we get
\[
\partial_t^k\bigl(\partial_t w_{u_{x_a}}+\gamma w_{u_{x_a}}\bigr)(0,\lambda)
=
\partial_t^k\bigl(\partial_t w_{u_{x_b}}+\gamma w_{u_{x_b}}\bigr)(0,\lambda).
\]
By the induction hypothesis \eqref{eq:claim-k}, this yields
\[
\partial_t^{k+1} w_{u_{x_a}}(0,\lambda)
=
\partial_t^{k+1} w_{u_{x_b}}(0,\lambda),
\quad \forall \lambda\in[0,1].
\]
Thus \eqref{eq:claim-k} holds for all $k=0,\dots,m-1$.

Evaluating \eqref{eq:claim-k} at $\lambda=1$ and using the boundary relation
$w_u(t,1)=u(t)$, we obtain
\[
u_{x_a}^{(k)}(0)=u_{x_b}^{(k)}(0),
\quad k=0,\dots,m-1.
\]
Since $u_x(s)=-h(X(x;s))$, this means
\[
L_f^k h(x_a)=L_f^k h(x_b),
\quad k=0,\dots,m-1.
\]
Therefore,
\[
H(x_a)=H(x_b).
\]
Assumption~\ref{diff_obs} then implies $x_a=x_b$, which proves the injectivity of $T_0$.
\end{proof}

\subsection{Injectivity of \texorpdfstring{$T$}{T}} 
To prove the injectivity of $T$, we will need Theorem~\ref{T0 inj} and the following Lemma
\begin{lem}\label{tv cap a}
    The backstepping transformation $\Tv$ defined in Section~\ref{Observer naive} and $\mathcal{A}$ defined in Lemma~\ref{pb stable heat bien posé} satisfies
    $$\Tv(\mathcal{A})\cap \mathcal{A} = \{0\}.$$
\end{lem}
\noindent The proof is postponed to~\ref{tech lemma 2}.
\begin{thm} \label{T inj}
    Under Assumptions~\ref{diff_obs} and~\ref{ass:bound}, the operator defined by
    \fonction{\Txv}{\mathcal{X}\times L^2(0, 1)}{L^2(0, 1)}{(x, v)}{\Tx(x) + \Tv v}
   is injective when restricted to $\mathcal{X}\times \mathcal{A}$.
\end{thm}

\begin{proof}
     Let $x_a, x_b\in \mathcal{X}$ and $v_a, v_b \in \mathcal{A}$ such that $T(x_a, v_a) = T(x_b,v_b)$ i.e for all $\lambda \in  [0,1]$, \begin{equation} \label{egalite preuve injectivite}T_0(x_a,\lambda)-T_0(x_b,\lambda) = \Tv(v_b-v_a).\end{equation}
Because of Definition~\ref{def T0}, we have
$$T_0(x_a,\cdot)-T_0(x_b,\cdot)\in\mathcal{A}.$$
Hence, using Lemma~\ref{tv cap a}, we obtain

$$T_0(x_a, \cdot) = T_0(x_b,\cdot)$$
and
$$\Tv(v_b-v_a) = 0.$$
From the injectivity of \( T_0 \) (Theorem~\ref{T0 inj}) and the invertibility of \( \Tv \), we can deduce respectively that \( x_a = x_b \) and \( v_a = v_b \).
\end{proof}

\begin{rem}
    This choice of space restriction is relevant as for all solution $v$ of the PDE~\eqref{eq:syst}, the parabolic regularization of the heat equation implies $v(t) \in \mathcal{A}$ for all $t>0$~\cite[Chapter 7]{evans10}.
\end{rem}

\subsection{Observer convergence}
\label{sec:main_result}
\modif{We have successfully constructed the Backstepping-KKL map \[
T(x,v)=T_0(x)+\Tv v.
\]
To construct our Backstepping-KKL observer we would need a uniformly continuous left inverse \[
T^{inv}:L^2(0,1)\to \mathbb R^n\times L^2(0,1)
\]
i.e. satisfying
\[
T^{inv}(T(x,v))=(x,v),
\quad
\forall (x,v)\in\mathcal X\times\mathcal A.
\]
\begin{claim}
    This is however not possible in the $L^2$ topology. 
\end{claim}
\begin{proof}
Recall that
\[
\mathcal A
=
\left\{
v\in L^2(0,1):
v \text{ is analytic and }
v^{(2k+1)}(0)=0,\ \forall k\in\mathbb N
\right\}.
\]
Then, $\mathcal A$ is dense in $L^2(0,1)$.
Indeed, for every $k\in\mathbb N$, the function
\[
e_k(\lambda):=\cos(k\pi\lambda)
\]
belongs to $\mathcal A$, since it is analytic and all its odd-order
derivatives vanish at $\lambda=0$. Since the family
$\bigl(\cos(k\pi\cdot)\bigr)_{k\in\mathbb N}$ is complete in
$L^2(0,1)$, it follows that $\mathcal A$ is dense in $L^2(0,1)$.
Since $\Tv$ is a bounded isomorphism of $L^2(0,1)$, we obtain that $\Tv(\mathcal A)$ is also dense in $L^2(0,1)$.
Choose $x_a,x_b\in\mathcal X$ with $x_a\neq x_b$.
Hence there exists a
sequence $(v_j)\subset\mathcal A$ such that
\[
\Tv v_j
\longrightarrow
T_0(x_a)-T_0(x_b)
\quad\text{in }L^2(0,1).
\]
Therefore
\[
T(x_b,v_j)
=
T_0(x_b)+\Tv v_j
\longrightarrow
T_0(x_a)
=
T(x_a,0)
\quad\text{in }L^2(0,1).
\]
Suppose that a continuous left inverse $T^{inv}$ existed on
$T(\mathcal X\times\mathcal A)$. Then
\[
T^{inv}(T(x_b,v_j))
=
(x_b,v_j)
\]
would converge in $\mathbb R^n\times L^2(0,1)$ to
\[
T^{inv}(T(x_a,0))
=
(x_a,0).
\]
In particular, the first component would satisfy
\[
x_b\longrightarrow x_a,
\]
which is impossible because $x_a\neq x_b$ is fixed. \end{proof}
As explained in the proof of Theorem~\ref{main result}, the solution to obtain a left inverse of $T$ is to consider $T(\mathcal X \times \mathcal A)\subset H^\infty(0,1)$ on a stronger topology than the $L^2$ one.
Define
\[
\Jinf:H^{\infty}([0,1])\to\R^{\N},
\quad
\Jinf w
:=
\bigl(w'(0),w^{(3)}(0),w^{(5)}(0),\ldots\bigr).
\]
We equip \(\R^{\N}\) with the product topology, metrized for instance by
\[
\dprod(a,b)
:=
\sum_{k=0}^{\infty}
2^{-(k+1)}
\min\{1,|a_k-b_k|\}.
\]
Thus
\[
\dprod(a^j,a)\to0
\quad\Longleftrightarrow\quad
 a_k^j\to a_k
\ \text{for all fixed }k.
\]
For $z_1,z_2\in H^\infty(0,1)$, define
\begin{equation}
\label{eq:dstar}
d_\star(z_1,z_2)
:=
\|z_1-z_2\|_{L^2}
+
d_\Pi
\left(
J_\infty\Tv^{-1}z_1,
J_\infty\Tv^{-1}z_2
\right).
\end{equation}
Hence, the convergence of $\zhat$ solution of~\eqref{KKL-target system} towards $z=T(x,v)$ (with $(x,v)$ a trajectory of~\eqref{eq:syst}) will be established for the metric $\dstar$.
}
\begin{thm} \label{main result}
Under Assumption~\ref{ass:bound}, the Backstepping-KKL transformation $T$ from Theorem~\ref{T inj} is well-defined and surjective. Additionally, under the differential observability Assumption~\ref{diff_obs}, $T$ is injective when restricted to $\mathcal{X} \times \mathcal{A}$. 

\modif{Moreover, under the additional assumption that $\mathcal X$ is compact, it admits a uniformly continuous left inverse \[T^{-1}:T(\mathcal X\times \mathcal A)\to \R^n\times L^2(0,1),\] 
where $T(\mathcal X\times \mathcal A)$ is equipped with $d_\star$ defined in~\eqref{eq:dstar}.
Then, for all $x_0\in\mathcal X_0$, $v_0\in L^2(0,1)$ and
$\hat z_0\in L^2(0,1)$, let $(x,v)$ denote the solution of
\eqref{eq:syst} with initial condition $(x_0,v_0)$, and let $\hat z$
denote the solution of \eqref{KKL-target system} with initial
condition $\hat z_0$.
For $t>0$, choose
\begin{equation}\label{eq:xhat_choice}
\hat x(t)\in
\argmin_{\xi\in\mathcal X}
d_\Pi\bigl(J_\infty\Tv^{-1}\hat z(t),J_\infty\Tv^{-1}T_0(\xi)\bigr),
\end{equation}
and set
\[
\hat v(t)
:=
\Tv^{-1}\bigl(\hat z(t)-T_0(\hat x(t))\bigr).
\]
Then
\[
|\hat x(t)-x(t)|
+
\|\hat v(t)-v(t)\|_{L^2}
\underset{t\to+\infty}{\longrightarrow}0.
\]}
\end{thm}
\modif{The proof is postponed~\ref{appendix:main_result}.
The choice of $\xhat$ in~\eqref{eq:xhat_choice} is motivated by the
following identity. Let
\[
z(t)=T_0(x(t))+\Tv v(t), \quad t>0,
\]
where $(x,v)$ is a trajectory of~\eqref{eq:syst}. Since
$v(t)\in\mathcal A$ for all $t>0$, we have
$
\Jinf v(t)=0.
$
Applying $\Jinf\Tv^{-1}$ to $z(t)$ therefore gives
\[
\Jinf\Tv^{-1}z(t)
=
\Jinf\Tv^{-1}T_0(x(t)).
\]
Thus, $\Jinf\Tv^{-1}z(t)$ contains precisely the information used to
recover the ODE state $x(t)$ through the map \[
\Psi:\mathcal X\to\R^\N,
\quad
\Psi(x)
:=
J_\infty\Tv^{-1}T_0(x).
\]
\begin{rem}
 The obstruction for the existence of a uniformly continuous left inverse $T^{inv}$ in the $L^2$ topology comes from the density of
$\mathcal A$, and consequently of $\Tv(\mathcal A)$, in
$L^2(0,1)$. This obstruction disappears for the stronger topology
induced by $d_\star$. Indeed, for all $v\in\mathcal A$,
$
J_\infty v=0,
$
so that every $z\in\Tv(\mathcal A)$ satisfies
$
J_\infty\Tv^{-1}z=0.
$
Consequently, a sequence in $\Tv(\mathcal A)$ cannot converge with
respect to $d_\star$ to a profile $z$ such that
\[
J_\infty\Tv^{-1}z\neq0.
\]
Thus the additional derivative information prevents the
density phenomenon responsible for the discontinuity of the inverse
in the $L^2$ topology. 
\end{rem}}
\begin{rem}\label{rem:gamma_tuning}
\modif{Although the target error $e(t)$ converges exponentially to zero in
$L^2(0,1)$, this does not directly imply exponential convergence of the
observer states $(\hat x,\hat v)$ towards $(x,v)$. Indeed, the
reconstruction of the ODE state relies on the inverse map
$\Psi^{-1}$, which is only known to be uniformly continuous on
$\Psi(\mathcal X)$, and no quantitative modulus of continuity is
assumed. Moreover, the convergence in the stronger topology induced by
$d_\star$ is established asymptotically, without an explicit
exponential rate. Consequently, the present result guarantees
asymptotic, but not necessarily exponential, convergence of the
observer states. For further developments concerning the convergence
speed of KKL observers, we refer to~\cite{andrieuspeedy}.}

\end{rem}
 \modif{\begin{rem}
As a by-product of Theorem~\ref{main result}, we also obtain an
infinite-dimensional observer for the ODE
\[
\dot x=f(x),\qquad y=h(x).
\]
Indeed, under Assumption~\ref{diff_obs}, the map
\[
T_0:\mathcal X\to L^2(0,1)
\]
is injective. Moreover, if $\mathcal X$ is compact, then $T_0$ is
continuous and therefore
\[
T_0^{-1}:T_0(\mathcal X)\to\mathcal X
\]
is uniformly continuous.
Let $\hat z$ be the solution of~\eqref{eq:z_ode} with arbitrary
initial condition $\hat z_0\in L^2(0,1)$, and define
\[
\hat x(t)\in
\argmin_{\xi\in\mathcal X}
\|\hat z(t)-T_0(\xi)\|_{L^2}.
\]
Then, for every $x_0\in\mathcal X_0$, if $x$ denotes the solution of
$\dot x=f(x)$ with initial condition $x_0$, one has
\[
|\hat x(t)-x(t)|
\underset{t\to+\infty}{\longrightarrow}0.
\]
\end{rem}}

\begin{algorithm}[t]
\caption{Numerical inversion of the Backstepping-KKL map}
\label{alg:numerical_inverse}

\textbf{Input:} Current observer state $\hat z(t,\cdot)$, current
measurement $y(t)$, KKL map $T_0$, inverse backstepping operator
$\Tv^{-1}$, and a finite-dimensional numerical characterization of
the admissible class $\mathcal A$.

\medskip

1. For each candidate ODE state $x$, define the associated candidate
PDE profile
\begin{equation}
\label{eq:v_candidate_general}
    v^x
    :=
    \Tv^{-1}
    \bigl(
        \hat z(t,\cdot)-T_0(x,\cdot)
    \bigr).
\end{equation}

2. Construct a finite collection of residuals
\begin{equation}
\label{eq:general_residuals_inverse}
    R_j(x;\hat z,y),
    \quad
    j=1,\ldots,N_c,
\end{equation}
measuring the violation by $v^x$ of the conditions characterizing
$\mathcal A$, together with the available boundary information, in
particular
\[
    v^x(0)=y(t).
\]

3. Compute the ODE-state estimate as
\begin{equation}
\label{eq:general_numerical_inverse_x}
    \hat x(t)
    =
    \arg\min_{x\in X}
    \sum_{j=1}^{N_c}
    \left|
        R_j(x;\hat z(t,\cdot),y(t))
    \right|^2,
\end{equation}
where $X$ denotes the prescribed compact region of interest for the
ODE state.

4. Reconstruct the PDE state from the inverse backstepping
transformation,
\begin{equation}
\label{eq:general_numerical_inverse_v}
    \hat v(t,\cdot)
    =
    \Tv^{-1}
    \left(
        \hat z(t,\cdot)
        -
        T_0(\hat x(t),\cdot)
    \right).
\end{equation}

\textbf{Output:}
\[
    \bigl(\hat x(t),\hat v(t,\cdot)\bigr).
\]

\end{algorithm}
\section{Numerical simulations} 
In the following section, we consider examples of cascades that illustrate the convergence of our Backstepping-KKL observer. To do so, we consider specific choices of the functions \(f\) and \(h\) that allow us to compute \(T_0\) explicitly by solving a heat equation along the trajectories of the ODE in the cascade~\eqref{eq:syst} or more generally by solving~\eqref{T_0 exemple} for all $x\in \mathcal{X}$, and all $\lambda \in [0,1]$,

\begin{equation} \label{T_0 exemple}
\left\{
\begin{aligned}
    &\partial_x \Tx(x, \lambda)f(x)= \partial_{\lambda}^2\Tx(x, \lambda) - \gamma T_0(x,\lambda)
    \\
    &\partial_{\lambda}\Tx(x, 0) = 0
    \\
    &\Tx(x, 1) = -h(x).
\end{aligned}
\right.
\end{equation}
Some numerical simulations accompany these examples. We use finite differences to compute the solutions of the cascade~\eqref{eq:syst} and the KKL-target system~\eqref{KKL-target system}. The Backstepping-KKL map $T$ is approximated using an integral discretization method for the backstepping part and an evaluation on a grid of the real solution $T_0$. The observer is constructed by inverting the Backstepping-KKL mapping $T$.
\modif{The numerical inversion used in the simulations is based on
Algorithm~\ref{alg:numerical_inverse} and can be viewed as a
finite-dimensional approximation of the reconstruction introduced in
Theorem~\ref{main result}. For a candidate ODE state $x$,
By linearity of $\Tv^{-1}$,
\[
    J_\infty v^x
    =
    J_\infty\Tv^{-1}\hat z
    -
    J_\infty\Tv^{-1}T_0(x).
\]
Hence, the theoretical reconstruction of the ODE state amounts to
selecting $x$ so that the odd boundary derivatives of $v^x$ vanish,
that is, so that $v^x\in \mathcal A$. In practice, evaluating arbitrarily high-order
derivatives is neither possible nor numerically robust. We therefore
replace this infinite collection of conditions by a finite-dimensional
set of residuals, together with the available boundary information
$v^x(0)=y(t)$. The estimate $\hat x$ is obtained by minimizing these
residuals, which provides a regularized finite-dimensional
approximation of the minimization problem~\eqref{eq:xhat_choice}.
}
\label{numerics}
\subsection{Example 1: parameter estimation}
When the functions $f$ and $h$ are defined by $f(x) = 0$ and $h(x) = x$ for all $x \in \R$, the system~\eqref{eq:syst} reduces to a parameter identification problem for the underlying PDE, which is of independent interest in its own right (see e.g~\cite{adaptive_heat_obs}). Indeed, the observer $\xhat$ will converge to $v(t,1) =x(t) = x\in \mathbb{R}$ (the parameter to be estimated). 
The differential observability Assumption~\eqref{diff_obs} is satisfied in this case as the function $h$ is injective.
We now determine \( T_0 \), the solution of~\eqref{T_0 exemple}, by explicitly seeking a solution of the form
\[
T_0(x,\lambda) = a(\lambda) x.
\]
Injecting this ansatz in equation~\eqref{T_0 exemple}, we obtain the following ODE satisfied by $a(\cdot)$:
\begin{equation} \label{a ex 1}
\left\{
\begin{aligned}
   a''(\lambda) &= \gamma a(\lambda)
    \\
    a'(0) &= 0
    \\
    a(1)&=-1.
\end{aligned}
\right.
\end{equation}
One can verify that the solution to the ODE~\eqref{a ex 1} is
$a(\lambda):=-\frac{\cosh(\lambda\sqrt{\gamma})}{\cosh(\sqrt{\gamma})}$.
The following numerical simulations are obtained in the case $\alpha = 0.5$, $\gamma = 1$ and $x(t) = x =1$.
Figure~\ref{x and xhat} displays the evolution of the state $x$ together with its estimate $\xhat$. 
Furthermore, Figures~\ref{error x} and~\ref{error v} illustrate the $L^2$-norm of the estimation error for the ODE and the PDE components, respectively.
As we can see in Figure~\ref{x and xhat}-\ref{error v}, the observer $(\xhat, \vhat)$ converges exponentially fast towards the state $(x,v)$ of~\eqref{eq:syst}. Convergence was expected because of Theorem~\ref{main result}.  Nevertheless, the errors remain near $10^{-3}$ after $t\approx 2$. This residual discrepancy can be attributed primarily to the static error introduced when inverting $T$, as well as to the spatial and temporal discretization inherent in the numerical approximation of the dynamics.

\begin{figure}[t]\label{fig:ex1}
    \centering
\begin{subfigure}[b]{0.45\textwidth}
    \centering
    \includegraphics[width=\linewidth]{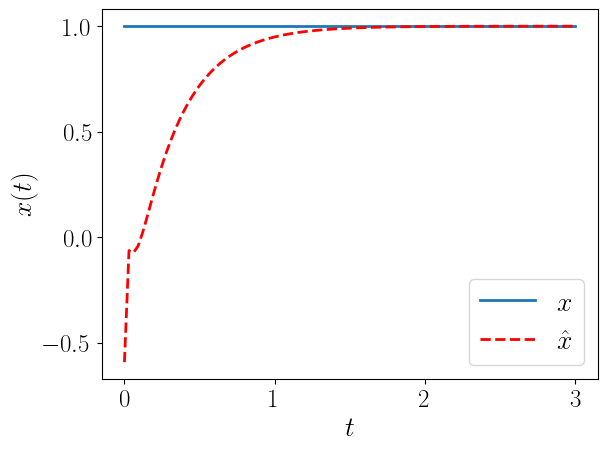}
    \caption{The ODE state $x$ and its estimation $\xhat$.}
    \label{x and xhat}
\end{subfigure}
\\
\medskip
\begin{subfigure}[b]{0.45\textwidth}
    \centering
    \includegraphics[width=\linewidth]{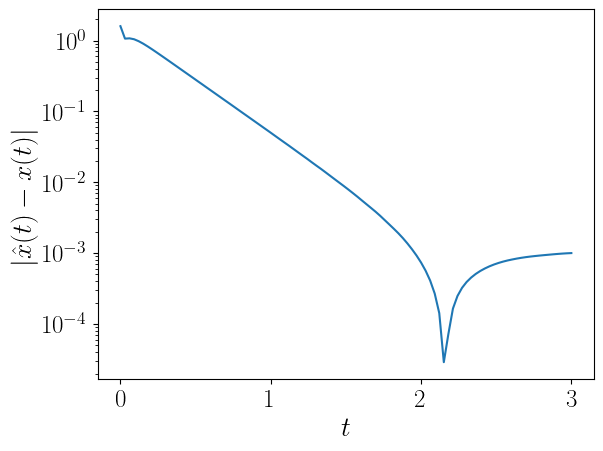}
    \caption{Absolute error between the ODE state $x$ and its estimation $\xhat$.}
    \label{error x}
\end{subfigure}
\quad
\begin{subfigure}[b]{0.45\textwidth}
    \centering
    \includegraphics[width=\linewidth]{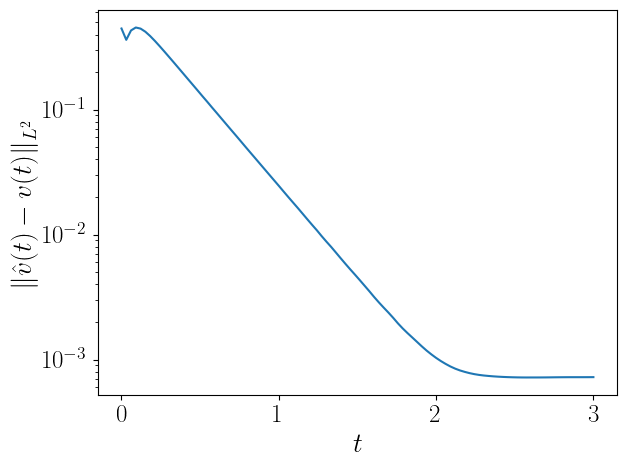}
    \caption{Error in $L^2$-norm between the PDE state $v$ and its estimation $\vhat$.}
    \label{error v}
\end{subfigure}
\end{figure}




\subsection{Example 2: harmonic oscillator with unknown frequency}
\modif{Let us now consider the unknown-frequency harmonic oscillator as in~\cite{praly2006new},
\begin{equation}
\label{eq:unknown_frequency_oscillator}
    f(x)
    =
    \begin{pmatrix}
        x_2\\
        -x_3^2 x_1\\
        0
    \end{pmatrix},
    \quad
    h(x)=x_1,
\end{equation}
where $x_3=\omega>0$ denotes the unknown constant frequency.
We consider the state on a set
\begin{equation} \label{eq:X_set_example2}
    \mathcal X
    \subset
    \left\{
        x\in\mathbb R^3
        \; ; \;
        x_1^2+x_2^2>0,\;
        \omega_{\min}\leq x_3\leq\omega_{\max}
    \right\},
\end{equation}
with $0<\omega_{\min}<\omega_{\max}$.
The exclusion of $(x_1,x_2)=(0,0)$ is necessary, since in this case
the measured output is identically zero for all value of the
frequency and the latter cannot be observed.
This example is particularly interesting since the unknown frequency
enters nonlinearly in the dynamics and has to be reconstructed together
with the oscillatory state. The differential observability
Assumption~\eqref{diff_obs} is satisfied on $X$. Indeed, as shown in~\cite{praly2006new},
the differential observability map
\begin{equation}
\label{eq:H_unknown_frequency}
    H(x)
    =
    \begin{pmatrix}
        h(x)\\
        L_fh(x)\\
        L_f^2h(x)\\
        L_f^3h(x)
    \end{pmatrix}
    =
    \begin{pmatrix}
        x_1\\
        x_2\\
        -x_3^2x_1\\
        -x_3^2x_2
    \end{pmatrix}
\end{equation}
is injective on $X$~\eqref{eq:X_set_example2}. 
}
\modif{The admissible class $\mathcal A$ is characterized by the conditions
\begin{equation}
\label{eq:A_odd_derivatives}
    v^{(2k+1)}(0)=0,
    \quad k\in\mathbb N,
\end{equation}
together with the measured boundary value
\[
    v(t,0)=y(t).
\]
For Example 2, rather than evaluating high-order spatial derivatives using finite differences, which
is numerically sensitive, we use several spatial samples of the
current candidate profile $v^x$ in a neighborhood
$[0,\lambda_{\rm fit}]$ of $\lambda=0$. On this interval, $v^x$ is
approximated by a polynomial
\begin{equation}
\label{eq:local_polynomial_A}
    v^x(\lambda)
    \simeq
    \sum_{j=0}^{N_p}
    a_j(x)\,
    \left(
        \frac{\lambda}{\lambda_{\rm fit}}
    \right)^j,
\end{equation}
with $N_p>1$ and odd.
For a function belonging to $\mathcal A$, the Taylor expansion at
$\lambda=0$ contains only even powers. Accordingly, a
finite-dimensional numerical measure of membership in $\mathcal A$ is
obtained from the odd polynomial coefficients
\[
    a_1(x),\quad
    a_3(x),\quad
    \ldots,\quad
    a_{N_p}(x).
\]
The residual vector used in Example~2 is therefore of the form
\begin{equation}
\label{eq:residual_example2}
    \mathcal R(x;\hat z,y)
    =
    \begin{pmatrix}
        v^x(0)-y
        \\
        a_1(x)
        \\
        a_3(x)
        \\
        \vdots
        \\
        a_{N_p}(x)
    \end{pmatrix},
    \quad
    v^x
    =
    \Tv^{-1}
    \bigl(
        \hat z-T_0(x)
    \bigr).
\end{equation}
In the simulations, a polynomial of degree $9$ is fitted to the current
spatial profile on a neighborhood of $\lambda=0$, and the coefficients
associated with the powers $1$, $3$, $5$, and $7$ are retained in
\eqref{eq:residual_example2}. The polynomial fit is therefore not
intended as a global approximation of $\mathcal A$, but as a
regularized finite-dimensional test of the odd-derivative conditions
in~\eqref{eq:A_odd_derivatives}.
In Example 2, straightforward computations shows that the solution to~\eqref{T_0 exemple} is, for all fixed $\omega$,
\begin{equation*}
\label{eq:T0_linear_example2}
    T_0(x,\lambda)
    =
    g_1(\omega,\lambda)x_1
    +
    g_2(\omega,\lambda)x_2,
\end{equation*}
where
\begin{align*}
    g_1(\omega,\lambda)
    &=
    -\Re\left(
        \frac{
        \cosh\!\left(
            \lambda\sqrt{\gamma+i\omega}
        \right)}
        {
        \cosh\!\left(
            \sqrt{\gamma+i\omega}
        \right)}
    \right),
    \\
    g_2(\omega,\lambda)
    &=
    -\frac{1}{\omega}
    \Im\left(
        \frac{
        \cosh\!\left(
            \lambda\sqrt{\gamma+i\omega}
        \right)}
        {
        \cosh\!\left(
            \sqrt{\gamma+i\omega}
        \right)}
    \right),
\end{align*}
where $\sqrt{\cdot}$ denotes the principal complex square root.
Since $\Tv^{-1}$ is linear, we also have
\[
    (\Tv^{-1}T_0(x))(\lambda)
    =
    G_1(\omega, \lambda)x_1
    +
    G_2(\omega, \lambda)x_2,
\]
with
\[
    G_j(\omega, \lambda)
    :=
    (\Tv^{-1}g_j(\omega))(\lambda),
    \quad j=1,2.
\]
Hence, for a fixed candidate value of $\omega$, the residual
\eqref{eq:residual_example2} is affine in $(x_1,x_2)$. The estimation
of $(x_1,x_2)$ can therefore be performed by a linear least-squares
problem. The unknown frequency is recovered by evaluating the
corresponding minimum residual over a prescribed compact interval
$[\omega_{\min},\omega_{\max}]$,
\begin{equation}
\label{eq:omega_projection_inverse}
    \hat\omega(t)
    =
    \argmin_{
        \omega\in[\omega_{\min},\omega_{\max}]
    }
    \;
    \min_{(x_1,x_2)}
    \left\|
        \mathcal R(
            (x_1,x_2,\omega);
            \hat z(t,\cdot),y(t)
        )
    \right\|^2.
\end{equation}
The resulting values
$\hat x_1(t)$, $\hat x_2(t)$, and $\hat\omega(t)$ define
$\hat x(t)$, after which the PDE-state estimate is computed from
\eqref{eq:general_numerical_inverse_v}.
}

\modif{We present below the numerical results obtained for the parameter values $\alpha = 0$, $\gamma =3$ and $x_0=(1,0, 1.2)$. Figure~\ref{figexample_2_dynamic_estimation.pdf} displays the evolution of the state $x$ together with its estimate $\xhat$. 
Furthermore, Figures~\ref{figerror_x_xhat_example2.pdf} and~\ref{figerror_v_vhat_example2.pdf} illustrate the $L^2$-norm of the estimation error for the ODE and the PDE components, respectively. Finally, Figures~\ref{figerror_inversion_zhat_gamma05}-\ref{figerror_inversion_zhat_gamma10} depict the $L^2$-norm of the estimation error in approximating the dynamics of $T(x,v)$ by $\zhat$, a solution of~\eqref{KKL-target system} and numerical error of the inversion of the Backstepping-KKL $T$.}

\begin{figure}[t]\label{fig:ex2}
    \centering
\begin{subfigure}[b]{0.45\textwidth}
    \centering
    \includegraphics[width=\linewidth]{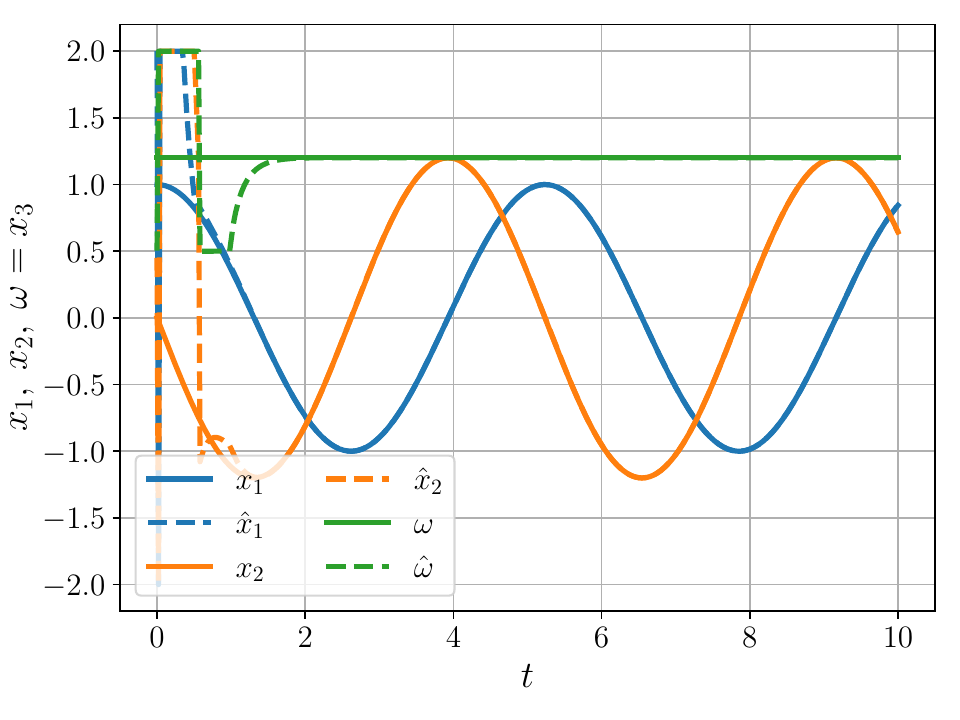}
    \caption{The ODE state $x$ and its estimation $\xhat$.}
    \label{figexample_2_dynamic_estimation.pdf}
\end{subfigure}
\\
\medskip
\begin{subfigure}[b]{0.45\textwidth}
    \centering
    \includegraphics[width=\linewidth]{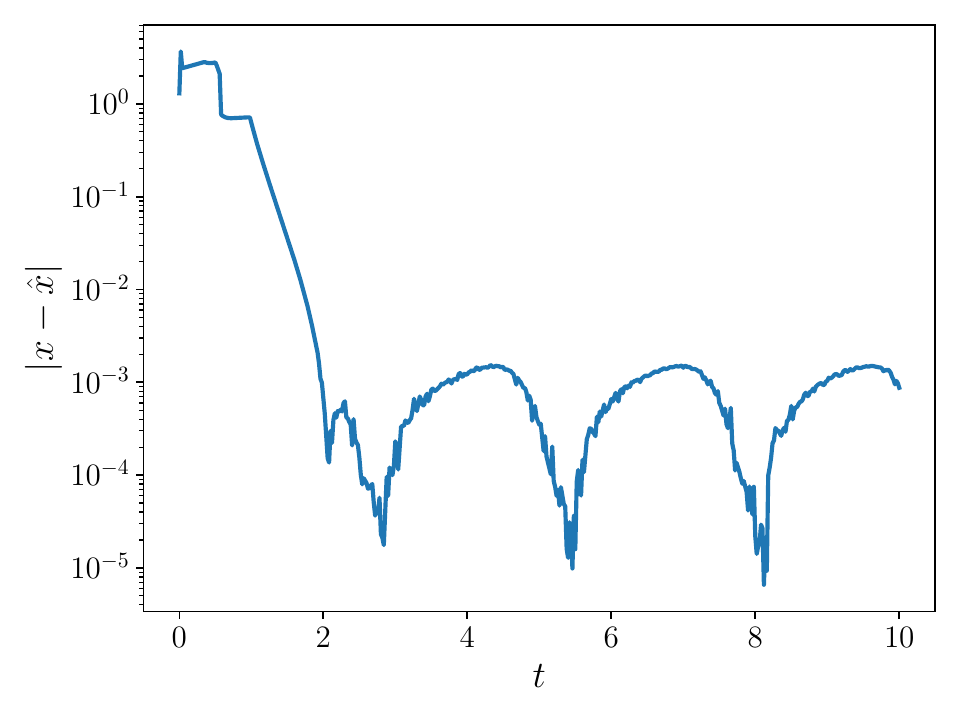}
    \caption{Absolute error between the ODE state $x$ and its estimation $\xhat$.}
    \label{figerror_x_xhat_example2.pdf}
\end{subfigure}
\quad
\begin{subfigure}[b]{0.45\textwidth}
    \centering
    \includegraphics[width=\linewidth]{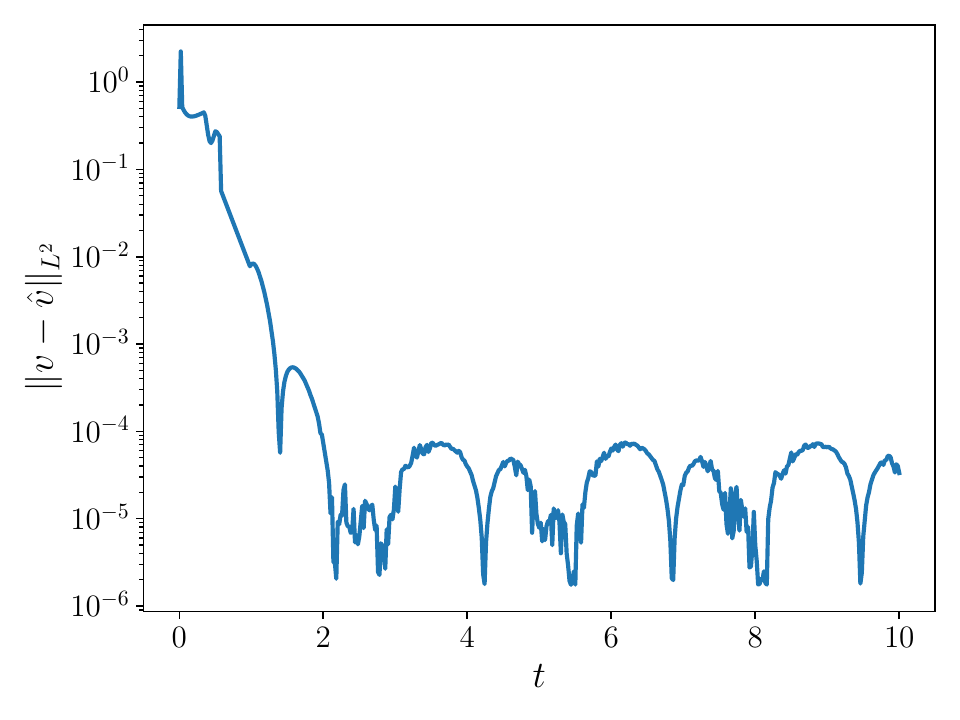}
    \caption{Error in $L^2$-norm between the PDE state $v$ and its estimation $\vhat$.}
    \label{figerror_v_vhat_example2.pdf}
\end{subfigure}
\end{figure}

As we can see in Figures~\ref{figexample_2_dynamic_estimation.pdf}-\ref{figerror_v_vhat_example2.pdf}, the observer $(\xhat, \vhat)$ converges exponentially fast towards the state $(x,v)$ of~\eqref{eq:syst}. Convergence was expected because of Theorem~\ref{main result}.  Nevertheless, the errors remain near $10^{-3}$ after $t\approx 2$. This residual discrepancy can be attributed primarily to the static error introduced when inverting $T$, as well as to the spatial and temporal discretization inherent in the numerical approximation of the dynamics.

In Figures~\ref{figerror_inversion_zhat_gamma3}-\ref{figerror_inversion_zhat_gamma10}, we report the error between $T(x(t), v(t))$ and $\zhat(t)$ for various values of the parameter $\gamma$. This error serves as a quantitative measure of the accuracy of our numerical approximation of the Backstepping-KKL mapping $T$. The observed oscillations, typically ranging between $10^{-5}$ and $10^{-7}$, can be regarded as the minimal discrepancy attainable between the true state $(x,v)$ and its estimate $(\xhat,\vhat)$ if the inversion was well conditioned.  

\begin{table}[t]
\centering
\caption{Sensitivity of the numerical observer to the tuning parameter $\gamma$
for Example~2.}
\label{tab:gamma_sensitivity}
\scalebox{1}{
\begin{tabular}{c|c|c|c}
\hline
$\gamma$
&
$t_{\mathrm{settle}}$
&
$|\hat{x}(10)-x(10)|$
&
$\|\hat{v}(10)-v(10)\|_{L^2}$
\\
\hline
$0.5$
& $2.10$
& $2.61e-03$
& $2.03e-04$
\\

$3$
& $1.14$
& $8.73e-04$
& $3.33e-05$
\\

$10$
& $0.46$
& $1.09e-03$
& $3.86e-05$
\\
\hline
\end{tabular}
}
\end{table}
\modif{To further quantify the influence of the tuning parameter \(\gamma\), Table~\ref{tab:gamma_sensitivity} reports several performance indicators for Example~2. The same initial conditions, spatial and temporal discretizations, and numerical inversion procedure are used for all tested values of \(\gamma\). The quantity  \(t_{\mathrm{settle}}\) denotes the first time after which this error \(\|\hat z-T(x,v)\|_{L^2}\) remains below \(10^{-3}\), i.e.,
$$t_{\mathrm{settle}} =\inf\{t>0: \|\hat z(\tau)-T(x(\tau),v(\tau))\|_{L^2}<10^{-3} ~\forall \tau>t\}.$$ 
The results illustrate the different effects of \(\gamma\) on the target dynamics and on the reconstruction of the original states. As expected, increasing \(\gamma\) accelerates the convergence in the transformed coordinates. This is consistent with the fact that \(\gamma\) directly increases the decay rate of the target dynamics.
However, this faster convergence of \(\hat z\) does not translate directly into improved reconstruction of the original states (see Remark~\ref{rem:gamma_tuning}). Among the tested values, \(\gamma=3\) provides the best compromise. 
These results emphasize that \(\gamma\) should not be chosen solely on the basis of the convergence rate of the target system. Although a larger value of \(\gamma\) yields faster convergence of \(\hat z\) towards \(T(x,v)\), both the Backstepping-KKL mapping and its numerical inverse depend on \(\gamma\). The reconstruction of \((x,v)\) may therefore become more difficult for some values of \(\gamma\).
}

\begin{figure}[t]\label{fig:ex3}
    \centering
\begin{subfigure}[b]{0.45\textwidth}
    \centering
    \includegraphics[width=\linewidth]{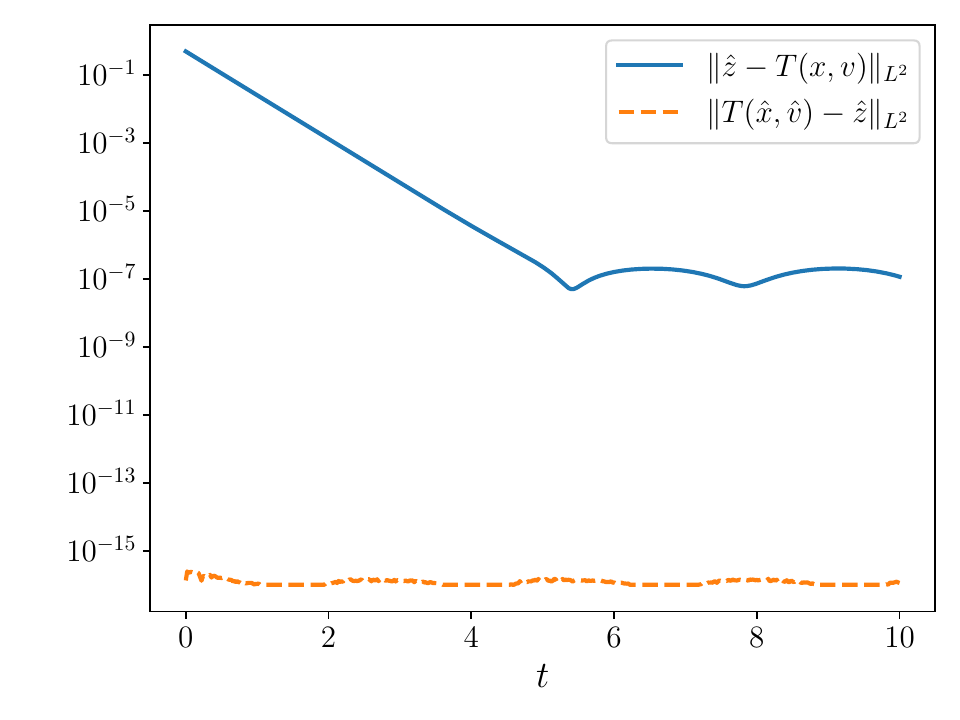}
   \caption{Error in $L^2$-norm between $\zhat$ and $T(x,v)$ and numerical inversion error for $\gamma=0.5$.}
    \label{figerror_inversion_zhat_gamma05}
\end{subfigure}
\\
\medskip
\begin{subfigure}[b]{0.45\textwidth}
    \centering
    \includegraphics[width=\linewidth]{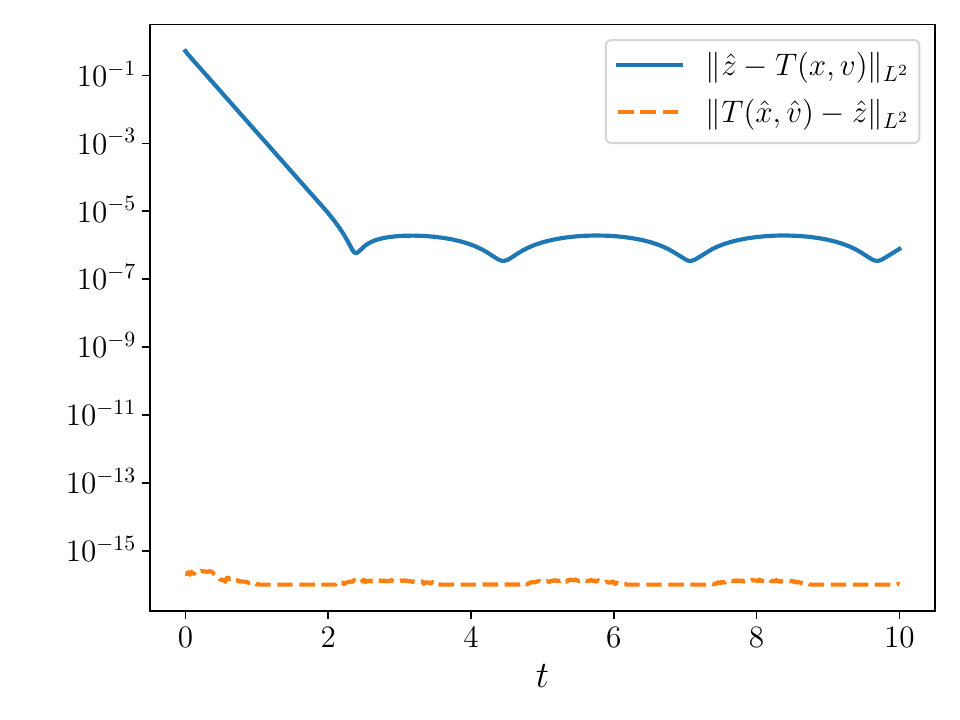}
    \caption{Error in $L^2$-norm between $\zhat$ and $T(x,v)$ and numerical inversion error for $\gamma=3$.}
    \label{figerror_inversion_zhat_gamma3}
\end{subfigure}
\quad
\begin{subfigure}[b]{0.45\textwidth}
    \centering
    \includegraphics[width=\linewidth]{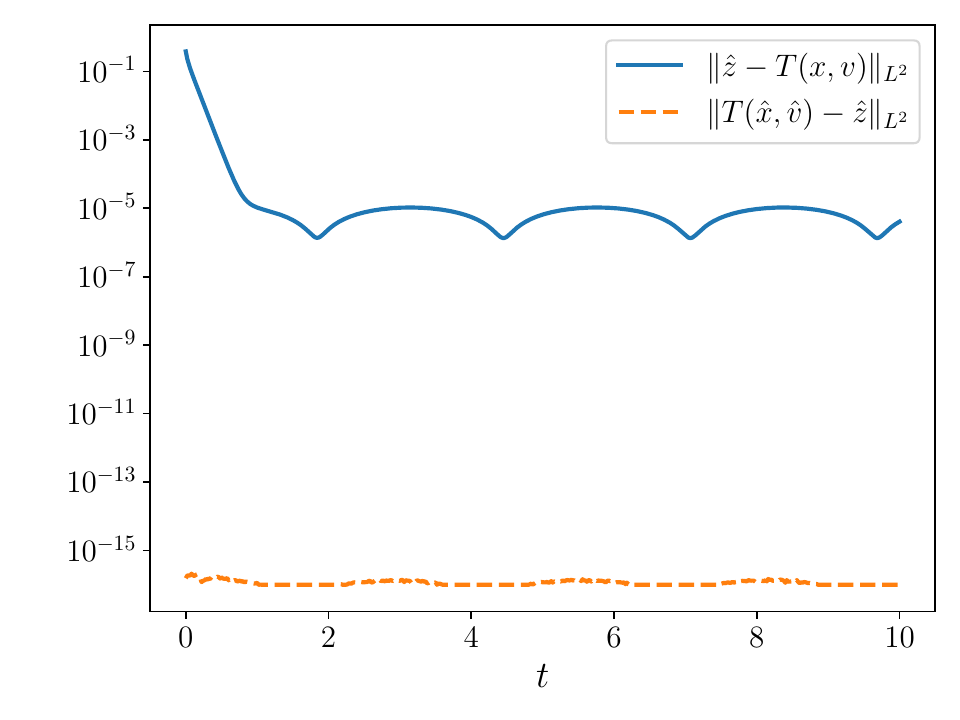}
    \caption{Error in $L^2$-norm between $\zhat$ and $T(x,v)$ and numerical inversion error for $\gamma=10$.}
    \label{figerror_inversion_zhat_gamma10}
\end{subfigure}
\end{figure}


\section{Conclusion}
In this paper, we have designed an observer for a cascade system composed of an ODE coupled with a one-dimensional heat equation. 

The Backstepping-KKL transformation~$T$ used in the observer design is constructed as the sum of a KKL mapping obtained by solving a steady-state problem without an initial condition, and a backstepping transformation for the one-dimensional heat equation. The mapping~$T$ is always surjective and becomes injective under a differential observability condition. \modif{Under the assumption that the trajectories of interests of the ODE remain in a compact set, $T$ admits a uniformly continuous left inverse, ensuring the convergence of the observer.}
As a direction for future work, we aim to extend our approach to cascade systems composed of an ODE and a general one-dimensional linear parabolic or hyperbolic PDE.
This work is also a first step towards the use of KKL observers for more general nonlinear infinite-dimensional systems.

\appendix

\section{Proof of Lemma~\ref{pb stable heat bien posé}}
\label{tech lemma 1}
    In order to establish the well-posedness of system~\eqref{stable heat}, we first prove the existence of a solution by introducing a candidate inspired 
by the solution of a related problem studied by Tikhonov. The uniqueness of the solution shall be proved using a Poincaré type inequality and a Lyapunov function.
Lastly, we prove for all $t\leq0$, $w_u(t,\cdot)\in \mathcal{A}$ directly using the explicit form of $w_u$.

\paragraph{Existence of a solution}
It has been shown in~\cite[Chapter III, section 3.4]{TikhonovSamarskii1990}
that for any bounded piecewise continuous function $\psi$, the function $w\in C^1((-\infty, 0], C^2([0,1]))$ defined for all $t \leq 0$ and for all $\lambda \in [0,+\infty)$ by $$w(t,\lambda) := \int_{-\infty}^t \psi(\tau)G(\lambda,t-\tau)d\tau$$
with \begin{equation*}G(\lambda, \tau) :=\frac{\lambda}{2\sqrt{\pi}\tau^{3/2}}\exp\left({\frac{-\lambda^2}{4\tau}}\right),\end{equation*}
is the unique bounded solution of the PDE
\begin{equation}\label{tick meta}\begin{cases}
    &\partial_t w(t,\lambda)= \partial_{\lambda}^2 w(t,\lambda)\\
   &w(t,0)=\psi(t),
\end{cases}\end{equation}
\modif{where we implicitly use the fact that $G(\cdot, \tau)$ is an approximate identity family, hence the boundary condition is to be understood as $w(t,0)=\lim_{\lambda \to 0}w(t,\lambda)=\psi(t)$.}
Hence, let us first introduce a candidate solution for problem~\eqref{stable heat inter} defined for all $t\leq 0$ and all $\lambda \in [0,1]$ by \begin{equation}\label{wu candidat}
    \modif{w_q(t,\lambda) := \int_{-\infty}^t \psi_q(\tau)(G(1-\lambda,t-\tau)+G(1+\lambda,t-\tau))d\tau},
\end{equation}
\modif{For all $\lambda \in [0, 1],$}
\begin{equation}
    \label{stable heat inter}
    \begin{cases}
    \partial_t w_q(t,\lambda) &= \partial_{\lambda}^2 w_q(t,\lambda) \\
    \partial_{\lambda}w_q(t,0)&=0 \\
    w_q(t,1) &= q(t),
\end{cases}
\end{equation}
where $\psi_q$ is a bounded piecewise continuous function to be defined and $q$ is a continuous bounded function on $(-\infty,0]$. The function $w_q$ is in $C^1((-\infty,0]), C^2([0,1]))$ and for all $t\leq0$, for all $\lambda \in [0,1]$, 
as \modif{$w_q(t,\lambda) =  w(t,1-\lambda) + w(t,1+\lambda)$ with $w$ solution of problem~\eqref{tick meta} with $\psi = \psi_q$, we have
\begin{align*}
    \partial_t w_q(t,\lambda) &= \partial_t(w(t,1-\lambda) + w(t,\lambda+1)) \\
    &=\partial_{\lambda}^2(w(t,1-\lambda) + w(t,\lambda+1)) \\
    &=\partial_{\lambda}^2 w_q(t,\lambda).
\end{align*}
Then, 
$$ \partial_\lambda w_q(t,0) = -\partial_\lambda w(t,1) + \partial_\lambda w(t,1) =0, $$}
Next, we find $\psi_q$ a piecewise bounded solution on $(-\infty,0]$ to fulfill the boundary condition $w_q(t,1) = q(t)$. 
\modif{We have 
$$ w_q(t,1) = \psi_q(t) + \int_{-\infty}^{t} \psi_q(\tau)G(2,t-\tau)\,d\tau . $$}
Replacing $w_q(t,1)$ with $q(t)$, we want to solve the following Volterra equation of the second kind with an infinite horizon,
\modif{\begin{equation}\label{volterra}q(t) = \psi_q(t)+\int_{- \infty}^t\psi_q(\tau)G(2,t-\tau)d\tau \quad \text{for all $t\leq0$.}\end{equation}}
According to~\cite[Theorem 4.5, Chapter 2]{GripenbergLondenStaffans1990}, as $G(2,\cdot)\in L^1(0, +\infty)$, equation~\eqref{volterra} admits a unique bounded piecewise continuous solution if and only if $1 +G(2,s)\neq 0$ for all $s\in \mathbb{C}$ with $\Re(s) \geq 0$ where $G(2,s)$ is the Laplace transform of $G(2,\cdot)$.
By~\cite[Theorem 2.8, Chapter 2]{GripenbergLondenStaffans1990}, the Laplace transform of $G(2,\cdot)\in L^1(0, +\infty)$ is well defined on $\{s\in \mathbb{C}, \Re(s) \geq0\}$. 
Let us compute this Laplace transform, for $s\in \{s\in \mathbb{C}, \Re(s) \geq 0\} $. Using~\cite[Formula 9 section 3.471 page 368]{gradshteyn2007table}, we get
\begin{align*}
    G(2,s) &= \frac{1}{\sqrt{\pi}}\int_0^{\infty} \frac{1}{t^{3/2}}e^{-(1/t +st)}dt \\
    &= \frac{1}{\sqrt{\pi}} 2s^{1/4}K_{1/2}(2\sqrt{s}) 
    \\&=e^{-2\sqrt{s}},
\end{align*}
where $K_{1/2}$ is the modified bessel function of the second kind of order $1/2$ (see \cite[Section 3.71 page 80, formula (13)]{WatsonBessel}).
Consequently, the condition  $1 +G(2,s)\neq 0$ for all $s\in \mathbb{C}$ with $\Re(s) \geq 0$ is fulfilled,
\modif{Moreover, Theorem~4.5 in~\cite{GripenbergLondenStaffans1990} also
ensures continuous dependence of the solution of~\eqref{volterra} on
$q$ in the $L^\infty$ norm. 
The existence of a bounded solution of~\eqref{stable heat inter} is proved.}
Then, a solution of system~\eqref{stable heat} is given by $w_u(t,\lambda) =e^{-\gamma t} w_q(t,\lambda)$ with $q(t) = e^{\gamma t}u(t).$
As $w_q$ is bounded (because $w$ is), there exists $C>0$ such that for all $t\leq0$ 
$$
\|w_u(t,\cdot)\|_{L^2(0,1)} \leq C e^{-\gamma t}.
$$
\modif{\paragraph{Uniqueness of the solution}
By linearity of problem~\eqref{stable heat}, it is sufficient to show that if
$w\in C^1((-\infty,0],C^2([0,1]))$ satisfies
\begin{equation}
\begin{cases}
    \partial_t w(t,\lambda)
    = \partial_\lambda^2 w(t,\lambda)-\gamma w(t,\lambda),\\
    \partial_\lambda w(t,0)=0,\\
    w(t,1)=0,
\end{cases}
\end{equation}
and if there exists $C>0$ such that
\begin{equation} \label{eq:growth_estimate}
    \|w(t,\cdot)\|_{L^2(0,1)}
    \leq C e^{-\gamma t},
    \quad t\leq 0,
\end{equation}
then $w\equiv 0$.
To this end, define
\[
    \overline w(t,\lambda):=e^{\gamma t}w(t,\lambda).
\]
The assumed growth estimate on $w$~\eqref{eq:growth_estimate} implies
\[
    \|\overline w(t,\cdot)\|_{L^2(0,1)}
    \leq C,
    \quad t\leq0.
\]
Moreover, $\overline w$ satisfies
\begin{equation}
\begin{cases}
    \partial_t \overline w(t,\lambda)
    =\partial_\lambda^2 \overline w(t,\lambda),\\
    \partial_\lambda \overline w(t,0)=0,\\
    \overline w(t,1)=0.
\end{cases}
\end{equation}
Let
\[
    V(t):=\frac12
    \|\overline w(t,\cdot)\|_{L^2(0,1)}^2.
\]
Integration by parts, together with the boundary conditions, yields
\[
\frac{d}{dt}V(t)
=
-\int_0^1
|\partial_\lambda\overline w(t,\lambda)|^2\,d\lambda.
\]
We now use the following Poincar\'e-type inequality. Since
$\overline w(t,1)=0$, for all $\lambda\in[0,1]$,
\[
\overline w(t,\lambda)
=
-\int_\lambda^1
\partial_\lambda\overline w(t,\widetilde\lambda)
\,d\widetilde\lambda.
\]
Hence, by the Cauchy--Schwarz inequality,
\[
|\overline w(t,\lambda)|^2
\leq
(1-\lambda)
\int_\lambda^1
|\partial_\lambda\overline w(t,\widetilde\lambda)|^2
\,d\widetilde\lambda.
\]
Integrating with respect to $\lambda$ and using Fubini's theorem gives
\begin{align*}
\int_0^1 |\overline w(t,\lambda)|^2\,d\lambda
&\leq
\int_0^1
|\partial_\lambda\overline w(t,\widetilde\lambda)|^2
\left(
\int_0^{\widetilde\lambda}(1-\lambda)\,d\lambda
\right)
d\widetilde\lambda\\
&=
\int_0^1
\left(
\widetilde\lambda-\frac{\widetilde\lambda^2}{2}
\right)
|\partial_\lambda\overline w(t,\widetilde\lambda)|^2
\,d\widetilde\lambda\\
&\leq
\frac12
\int_0^1
|\partial_\lambda\overline w(t,\lambda)|^2\,d\lambda.
\end{align*}
Consequently,
\[
    \frac{d}{dt}V(t)
    \leq -2
    \|\overline w(t,\cdot)\|_{L^2(0,1)}^2
    =-4V(t).
\]
Therefore, for all $t_0<t\leq0$,
\begin{equation}\label{V}
    0\leq V(t)
    \leq V(t_0)e^{-4(t-t_0)}.
\end{equation}
Since $\|\overline w(t,\cdot)\|_{L^2(0,1)}\leq C$ for all
$t\leq0$, we have $V(t_0)\leq C^2/2$. Thus,
\[
    0\leq V(t)
    \leq \frac{C^2}{2}e^{-4(t-t_0)}.
\]
For any fixed $t\leq0$, letting $t_0\to-\infty$ gives
\[
    V(t)=0.
\]
Since $t\leq0$ is arbitrary, $\overline w\equiv0$, and therefore
$w\equiv0$. This proves uniqueness.}
 \paragraph{Belonging to $\mathcal{A}$} \modif{We finally prove that, for all $t\leq 0$,
$w_u(t,\cdot)\in\mathcal A$.
Fix $t<0$ and choose $T<t$. Restricting the solution $w_u$
to the finite time interval $[T,0]$, we may regard
$w_u(T,\cdot)$ as an initial condition for the heat equation.
By parabolic regularization, the spatial profile
$\lambda\mapsto w_u(t,\lambda)$ is analytic for all $t>T$.
Since $T<t$ is arbitrary, this argument applies to every
$t<0$.}

Let $t\leq 0$ and $k\in\mathbb{N}$, let us prove $\partial_{\lambda}^{2k+1}w_u(t,0)=0$. Looking at the explicit form of $w_u$ given in equation~\eqref{wu candidat}, we only need to show for all $\tau \in (-\infty,t]$
$$\partial_{\lambda}^{2k+1}g(0, \tau)=0, $$
with
\modif{$$g(\lambda,\tau):=G(1-\lambda,t-\tau)+G(\lambda+1,t-\tau).$$Noticing that \( g \) is even in \( \lambda \), i.e.,
\[
g(-\lambda,\tau) = g(\lambda,\tau), \quad \text{for all } \lambda \in [0,1],
\]
and because every smooth and even function has all its odd-order derivatives vanishing at \( \lambda = 0 \), we can conclude the proof.}
\hfill $\blacksquare$

\section{Proof of Lemma~\ref{tv cap a}}
\label{tech lemma 2}

 Let $v\in \mathcal{A},~z\in \mathcal{A}$ such that \begin{equation}
        \label{eq zv}
        z= \Tv v
    \end{equation}
We want to show $v=z=0\in L^2(0,1)$.
  As both \( z \) and \( v \) are analytic functions, it is sufficient to prove that for all \( k \in \mathbb{N} \),
$$
 \frac{d^k z}{d\lambda^k} (0)
 = \frac{d^kv }{d\lambda^k} (0)
 = 0.
$$
Let us prove it by induction. For $k=0$, differentiating equation~\eqref{eq zv} we obtain for all $\lambda \in [0,1]$
$$\frac{dz}{d\lambda}(\lambda) =\frac{dv}{d\lambda}(\lambda) - p(\lambda, \lambda)v(\lambda)-\int_0^{\lambda} \partial_{\lambda}p(\lambda,\tilde{\lambda})v(\tilde{\lambda})d\tilde{\lambda}. $$
Evaluating the previous equation at $\lambda =0$, because $v,z\in \mathcal{A}$ and $p_{10}\neq0$ (Section~\ref{Observer naive}), we get
$$v(0)=0.$$
Using equation~\eqref{eq zv}, we obtain
$$z(0) = (\Tv v)(0) = v(0) = 0.$$
Suppose the property holds for $k-1\geq0$, let us prove it for $k$. 

If \( k \) is odd, the result follows directly from the definition of \( \mathcal{A} \), since \( v, z \in \mathcal{A} \) implies all odd-order derivatives vanish at \( \lambda = 0 \). If \( k \) is even, let us differentiate equation~\eqref{eq zv} \( k+1 \) times and evaluate at \( \lambda = 0 \). This gives
$$0=\frac{d^{k+1}z}{d\lambda^{k+1}}(0) =\frac{d^{k+1}v}{d\lambda^{k+1}}(0) + p_{10}\frac{d^{k}v}{d\lambda^{k}}(0). $$
Since \( \frac{d^{k+1} z}{d\lambda^{k+1}}(0) = 0 \) (as \( z \in \mathcal{A} \) and \( k+1 \) is odd), and \( p_{10} \neq 0 \), it follows that 
$$\frac{d^{k}v}{d\lambda^{k}}(0)=0.$$
Finally, differentiate equation~\eqref{eq zv} $k$ times then evaluate it at $\lambda = 0$
$$\frac{d^{k}z}{d\lambda^{k}}(0)=\frac{d^{k}v}{d\lambda^{k}}(0)=0.$$
This completes the induction, and hence the proof.
\hfill $\blacksquare$

\section{Proof of Theorem~\ref{main result}}
\label{appendix:main_result}
\modif{The surjectivity of
\[
T:\mathcal X\times L^2(0,1)\to L^2(0,1)
\]
follows directly from the surjectivity of the backstepping
transformation $\Tv$. Indeed, for any $z\in L^2(0,1)$ and any
$x\in\mathcal X$, setting
\[
v:=\Tv^{-1}\bigl(z-T_0(x)\bigr)
\]
gives
\[
T(x,v)=T_0(x)+\Tv v=z.
\]
The injectivity of $T$ on $\mathcal X\times\mathcal A$ has already
been established in Theorem~\ref{T inj}.
It remains to prove the existence of a uniformly continuous
left inverse for the topology induced by $d_\star$, and the
convergence of the resulting observer.
 \paragraph{Continuity and injectivity of $\Psi$.}
Recall that
\[
\Psi:\mathcal X\to\R^\N,
\quad
\Psi(x):=\Jinf\Tv^{-1}T_0(x).
\]
\begin{lem}
\label{lem:Psi_cont_inj}
Suppose that Assumption~\ref{ass:bound} holds and that
$\mathcal X$ is compact. Then $\Psi$ is continuous from
$\mathcal X$ into $(\R^\N,d_\Pi)$. Moreover, under
Assumption~\ref{diff_obs}, $\Psi$ is injective.
\end{lem}
\begin{proof}
We first prove continuity. 
As in the discussion before definition~\ref{def T0}, the vector field $f$ can be
modified outside a neighborhood of $\mathcal X$ so that all backward
trajectories involved in the definition of $T_0$ remain in a common
compact set.
Using the continuous dependence of the flow on its initial condition,
together with the uniform boundedness of $h(X(t;x))$ along the backward
trajectories, we obtain that
\[
x\longmapsto q_x: t\mapsto -e^{\gamma t}h(X(t,x))
\]
is continuous from $\mathcal X$ into $L^\infty((-\infty,0])$.
Let $\psi_{q_x}$ denote the solution of the Volterra
equation~\eqref{volterra}. Let $x_n\to x\in \mathcal X$, the continuous dependence of $\psi_{q_x}$ with respect to $q_x$ mentioned in
the proof of Lemma~\ref{pb stable heat bien posé} gives a constant
$C_\psi>0$ such that
\[
\|\psi_{q_{x_n}}-\psi_{q_x}\|_{L^\infty}
\leq
C_\psi
\|q_{x_n}-q_x\|_{L^\infty}.
\]
Consequently,
\begin{equation}
\label{eq:psi_q_continuity}
\|\psi_{q_{x_n}}-\psi_{q_x}\|_{L^\infty}
\longrightarrow0.
\end{equation}
Using the representation~\eqref{wu candidat} at $t=0$, we have
\[
T_0(x,\lambda)
=
\int_0^{+\infty}
\psi_{q_x}(-\tau)\,
\Gamma(\lambda,\tau)\,d\tau,
\]
where
\[
\Gamma(\lambda,\tau)
:=
G(1-\lambda,\tau)+G(1+\lambda,\tau).
\]
For all $j\in\N$, define
\[
K_j(\tau)
:=
\partial_\lambda^j\Gamma(0,\tau).
\]
Using the expression of $G$,
$$G(\lambda, \tau) :=\frac{\lambda}{2\sqrt{\pi}\tau^{3/2}}\exp\left({\frac{-\lambda^2}{4\tau}}\right),$$
we see that each $\partial_\lambda^j\Gamma(0,\tau)$ is a finite sum of terms of the form
\[
\tau^{-N}e^{-1/(4\tau)},
\]
for suitable $N>0$. Hence the Gaussian factor ensures integrability near
$s=0$, while the polynomial decay ensures integrability as
$s\to+\infty$. Hence,
\[
K_j\in L^1(0,+\infty)
\quad\text{for all }j\in\N.
\]
Therefore,
\[
\partial_\lambda^jT_0(x,0)
=
\int_0^{+\infty}
\psi_{q_x}(-s)K_j(s)\,ds,
\]
and~\eqref{eq:psi_q_continuity} implies
\begin{equation}
\label{eq:T0_derivative_continuity}
\partial_\lambda^jT_0(x_n,0)
\longrightarrow
\partial_\lambda^jT_0(x,0)
\quad
\forall j\in\N.
\end{equation}
In particular, the same representation also shows that
$T_0:\mathcal X\to L^2(0,1)$ is continuous.
By Lemma~\ref{lem:inverse_backstepping_regularity}, we can repeatedly differentiate the Volterra expression at $\lambda=0$ to obtain that, for all
$m\in\N$ there exist constants $c_{m,j}$ such that, for all
sufficiently regular $z$,
\[
\partial_\lambda^m(\Tv^{-1}z)(0)
=
\partial_\lambda^m z(0)
+
\sum_{j=0}^{m-1}
c_{m,j}\partial_\lambda^jz(0).
\]
Hence, for all fixed $k\in\N$, the $k$-th component of $\Psi$,
\[
\Psi_k(x)
=
\partial_\lambda^{2k+1}
\bigl(\Tv^{-1}T_0(x)\bigr)(0),
\]
is a finite linear combination of the quantities
$\partial_\lambda^jT_0(x,0)$. By~\eqref{eq:T0_derivative_continuity},
each map $\Psi_k:\mathcal X\to\R$ is continuous. This proves the continuity
of $\Psi$.
We now prove injectivity. Let $x_a,x_b\in\mathcal X$ satisfy
\[
\Psi(x_a)=\Psi(x_b),
\]
and set
\[
w
:=
\Tv^{-1}
\bigl(
T_0(x_a)-T_0(x_b)
\bigr).
\]
By Lemma~\ref{lem:inverse_backstepping_regularity}, the inverse
backstepping transformation preserves analyticity. Since
$T_0(x_a)-T_0(x_b)$ is analytic, $w$ is analytic as well.
Moreover, $\Psi(x_a)=\Psi(x_b)$ implies
\[
\partial_\lambda^{2k+1}w(0)=0,
\quad
\forall k\in\N.
\]
Thus
\[
w\in\mathcal A.
\]
On the other hand,
\[
\Tv w
=
T_0(x_a)-T_0(x_b).
\]
By Definition~\ref{def T0},
\[
T_0(x_a)-T_0(x_b)\in\mathcal A,
\]
and consequently
\[
\Tv w
\in
\Tv(\mathcal A)\cap\mathcal A.
\]
Lemma~\ref{tv cap a} gives
\[
\Tv w=0.
\]
Therefore
\[
T_0(x_a)=T_0(x_b),
\]
and Theorem~\ref{T0 inj} yields
\[
x_a=x_b.
\]
Hence $\Psi$ is injective.
\end{proof}
\paragraph{Uniformly continuous left inverse}
Since $\mathcal X$ is compact and $(\R^\N,d_\Pi)$ Lemma~\ref{lem:Psi_cont_inj} implies that
\[
\Psi:\mathcal X\to\Psi(\mathcal X)
\]
is a homeomorphism. Moreover, $\Psi(\mathcal X)$ is compact, so its
inverse
\[
\Psi^{-1}:\Psi(\mathcal X)\to\mathcal X
\]
is uniformly continuous.
As explained in the discussion after Theorem~\ref{main result}, the ODE state is uniquely recovered from $z$ by
\begin{equation}
\label{eq:x_inverse_main}
x
=
\Psi^{-1}
\left(
J_\infty\Tv^{-1}z
\right),
\end{equation}
and then
\begin{equation}
\label{eq:v_inverse_main}
v
=
\Tv^{-1}
\left(
z-T_0(x)
\right).
\end{equation}
Equations~\eqref{eq:x_inverse_main}--\eqref{eq:v_inverse_main}
therefore define the left inverse
\[
T^{-1}:
T(\mathcal X\times\mathcal A)
\to
\mathcal X\times L^2(0,1).
\]
Let now
\[
z_i=T(x_i,v_i),
\quad i=1,2.
\]
Since $v_i \in \mathcal A$, $\Jinf v=0$, hence
\begin{equation}\label{eq:min_exact_v}\Jinf \Tv^{-1}z_i = \Jinf \Tv^{-1}T_0(x_i= \Psi(x_i),\end{equation}
hence
\[
d_\Pi\bigl(\Psi(x_1),\Psi(x_2)\bigr)
=
d_\Pi\left(
J_\infty\Tv^{-1}z_1,
J_\infty\Tv^{-1}z_2
\right)
\leq
d_\star(z_1,z_2).
\]
The uniform continuity of $\Psi^{-1}$ therefore implies that
\[
d_\star(z_1,z_2)\to0
\quad\Longrightarrow\quad
|x_1-x_2|\to0
\]
uniformly with respect to $z_1,z_2$.
The continuity of $T_0:\mathcal X\to L^2(0,1)$, together with the
compactness of $\mathcal X$, implies that $T_0$ is uniformly
continuous. Furthermore,
\[
v_1-v_2
=
\Tv^{-1}
\left[
z_1-z_2
-
\bigl(T_0(x_1)-T_0(x_2)\bigr)
\right].
\]
Since $\Tv^{-1}$ is bounded on $L^2(0,1)$,
\[
\|v_1-v_2\|_{L^2}
\leq
\|\Tv^{-1}\|_{\mathcal L(L^2)}
\left(
\|z_1-z_2\|_{L^2}
+
\|T_0(x_1)-T_0(x_2)\|_{L^2}
\right).
\]
Consequently,
\[
T^{-1}:
\bigl(T(\mathcal X\times\mathcal A),d_\star\bigr)
\longrightarrow
\mathcal X\times L^2(0,1)
\]
is uniformly continuous when the arrival space is endowed with the
product norm
\[
|x_1-x_2|
+
\|v_1-v_2\|_{L^2}.
\]
\paragraph{Convergence in the stronger topology.}
Let $(x,v)$ be trajectories of~\eqref{eq:syst} and let
\[
z(t):=T(x(t),v(t))
\]
and define
\[
e(t):=\hat z(t)-z(t),
\]
with $\zhat $ a solution of~\eqref{KKL-target system}.
Since $\hat z$ and $z$ satisfy the same KKL target system with the
same measured output $y$, $e$ satisfies
\begin{equation}
\label{eq:error_main_result}
\left\{
\begin{aligned}
\partial_t e(t,\lambda)
&=
\partial_\lambda^2e(t,\lambda)-\gamma e(t,\lambda),
\\
\partial_\lambda e(t,0)&=0,
\\
e(t,1)&=0.
\end{aligned}
\right.
\end{equation}
Let
\[
A\phi:=-\partial_\lambda^2\phi+\gamma\phi,
\quad
D(A):=
\left\{
\phi\in H^2(0,1):
\partial_\lambda\phi(0)=0,\ \phi(1)=0
\right\}.
\]
Then~\eqref{eq:error_main_result} can be written as
\[
\dot e(t)=-Ae(t),
\quad
e(t)=e^{-tA}e(0).
\]
The operator $-A$ generates an analytic semigroup on $L^2(0,1)$. Moreover,
\[
\sigma(A)
=
\left\{
\left(k+\frac12\right)^2\pi^2+\gamma
\;:\;
k\in\mathbb N
\right\}.
\]
Since $\gamma>0$, we have $0\notin\sigma(A)$, and therefore
$0\in\rho(A)$. Hence, by
\cite[Theorem~6.13]{pazy1983semigroups}, for all integer $m\geq0$,
\[
\|A^m e(t)\|_{L^2}
\leq
M_m t^{-m}e^{-\omega t}\|e(0)\|_{L^2},
\quad t>0,
\]
for some $M_m,\omega>0$. In particular,
$\|e(t)\|_{L^2}\to0$ as $t\to+\infty$.
Moreover, for all $m\geq 1$, there exists a constant
$C_m>0$ such that
\[
\|u\|_{H^{2m}}
\leq
C_m\left(
\|A^m u\|_{L^2}
+
\|u\|_{L^2}
\right),
\quad
\forall u\in D(A^m).
\]
In particular, $D(A^m)$ is continuously embedded into
$H^{2m}(0,1)$. Therefore, for all
$\tau>0$, there exists $C_{m,\tau}>0$ such that
\begin{equation}
\label{eq:error_smoothing_main}
\|e(t)\|_{H^m}
\leq
C_{m,\tau}e^{-\omega t}\|e(0)\|_{L^2},
\quad t\geq\tau.
\end{equation}
Fix $k\in\N$ and choose
\[
m\geq 2k+ 2.
\]
The trace theorem~\cite[Theorem 8.2]{brezisfunctionnalanalysis}
gives that the point
evaluation at $\lambda=0$ is continuous on $H^1(0,1)$. Hence, by Lemma~\ref{lem:inverse_backstepping_regularity},
\[
\left|
\partial_\lambda^{2k+1}
\bigl(\Tv^{-1}e(t)\bigr)(0)
\right|\leq C_m \|\Tv^{-1}e(t)\|_{H^m}
\leq
C_k\|e(t)\|_{H^m}.
\]
Together with~\eqref{eq:error_smoothing_main}, this yields
\[
\partial_\lambda^{2k+1}
\bigl(\Tv^{-1}e(t)\bigr)(0)
\longrightarrow0 \quad \textit{as $t\to +\infty$} 
\]
for all fixed $k\in\N$. Hence
\[
J_\infty\Tv^{-1}e(t)\longrightarrow0  \quad \textit{as $t\to +\infty$} 
\]
in the product topology of $\R^\N$, that is,
\[
d_\Pi\left(
J_\infty\Tv^{-1}\hat z(t),
J_\infty\Tv^{-1}z(t)
\right)
\longrightarrow0.
\]
Combining this with the $L^2$ convergence of $e$ gives
\begin{equation}
\label{eq:dstar_error_main}
d_\star\bigl(\hat z(t),z(t)\bigr)
\longrightarrow0,
\quad \textit{as $t\to +\infty$}.
\end{equation}
for all positive time for which the quantities entering $d_\star$
are defined. This is ensured by the instantaneous regularization of
the heat equation~\cite[Chapter~7]{evans10} and
Lemma~\ref{lem:inverse_backstepping_regularity}.
\paragraph{Reconstruction of the ODE and PDE states.}
Although~\eqref{eq:dstar_error_main} holds, $\hat z(t)$ solution of~\eqref{KKL-target system} does not
necessarily belong to $T(\mathcal X\times\mathcal A)$ for finite
$t$. Thus the left inverse defined above cannot in general be applied
directly to $\hat z(t)$.
For $t>0$, define
\begin{equation}
\label{eq:xhat_main_result}
\hat x(t)
\in
\argmin_{\xi\in\mathcal X}
d_\Pi\bigl(J_\infty\Tv^{-1}\hat z(t),\Psi(\xi)\bigr).
\end{equation}
Such a minimizer exists because $\mathcal X$ is compact and $\Psi$
is continuous. We then define
\begin{equation}
\label{eq:vhat_main_result}
\hat v(t)
:=
\Tv^{-1}
\left(
\hat z(t)-T_0(\hat x(t))
\right).
\end{equation}
We have $v(t)\in \mathcal A$ for all $t>0$ due to the regularization of the heat equation~\cite[Chapter~7]{evans10}.
Therefore, as in~\eqref{eq:min_exact_v},
\[
J_\infty\Tv^{-1}z(t)
=
J_\infty
\left(
\Tv^{-1}T_0(x(t))+v(t)
\right)
=
\Psi(x(t)).
\]
Hence~\eqref{eq:dstar_error_main} implies
\begin{equation}
\label{eq:eta_Psi_main}
d_\Pi\bigl(
J_\infty\Tv^{-1}\hat z(t),\Psi(x(t))
\bigr)
\longrightarrow0.
\end{equation}
By the minimizing property of $\hat x(t)$,
\[
d_\Pi\bigl(
J_\infty\Tv^{-1}\hat z(t),\Psi(\hat x(t))
\bigr)
\leq
d_\Pi\bigl(
J_\infty\Tv^{-1}\hat z(t),\Psi(x(t))
\bigr).
\]
The triangle inequality then gives
\[
d_\Pi\bigl(
\Psi(\hat x(t)),\Psi(x(t))
\bigr)
\leq
2
d_\Pi\bigl(
J_\infty\Tv^{-1}\hat z(t),\Psi(x(t))
\bigr)
\longrightarrow0.
\]
Since $\Psi^{-1}$ is uniformly continuous on the compact set
$\Psi(\mathcal X)$,
\[
|\hat x(t)-x(t)|
\longrightarrow0.
\]
Finally, using
\[
v(t)
=
\Tv^{-1}
\left(
z(t)-T_0(x(t))
\right),
\]
together with~\eqref{eq:vhat_main_result}, we obtain
\[
\hat v(t)-v(t)
=
\Tv^{-1}
\left[
e(t)
+
T_0(x(t))-T_0(\hat x(t))
\right].
\]
Therefore
\[
\|\hat v(t)-v(t)\|_{L^2}
\leq
\|\Tv^{-1}\|_{\mathcal L(L^2)}
\left(
\|e(t)\|_{L^2}
+
\|T_0(x(t))-T_0(\hat x(t))\|_{L^2}
\right).
\]
The first term converges to zero by exponential stability of
\eqref{eq:error_main_result}. The second term converges to zero
because $T_0$ is uniformly continuous on the compact set
$\mathcal X$.
Thus
\[
\left|\hat x(t)-x(t)\right|
+
\|\hat v(t)-v(t)\|_{L^2}
\longrightarrow0,
\quad t\to+\infty,
\]
which concludes the proof.}
\hfill$\blacksquare$
\bibliographystyle{elsarticle-num}
\bibliography{references}

\end{document}